\documentclass[11pt,a4paper]{amsart}
\usepackage[utf8]{inputenc}
\usepackage{amsmath}
\usepackage{amsfonts}
\usepackage{amssymb}
\usepackage{amsthm}
\usepackage{stmaryrd}
\usepackage{graphicx}
\usepackage{tikz}
\usepackage{tikz-cd}
\usepackage{mathtools}
\usepackage{comment}
\usepackage{import}
\usepackage[
bookmarks         = true,
bookmarksnumbered = true,
pdfstartview      = FitH,
pdfpagelayout     = OneColumn,
pdfpagemode=UseOutlines,
colorlinks        = true,
linkcolor         = blue,
urlcolor          = magenta,
citecolor         =red,
backref           = page
]{hyperref}

\usepackage[draft,inline,nomargin,index]{fixme}

\usepackage{fourier-orns}
\usepackage{ mathrsfs }
\DeclareMathOperator{\Aut}{Aut}
\DeclareMathOperator{\Stab}{Stab}
\DeclareMathOperator{\Isom}{Isom}
\DeclareMathOperator{\Res}{Res}
\DeclareMathOperator{\pr}{pr}
\DeclareMathOperator{\Out}{Out}
\renewcommand{\d}{\mathrm d}

\renewcommand{\le}{\leqslant}
\renewcommand{\ge}{\geqslant}
\renewcommand{\epsilon}{\varepsilon}
\renewcommand{\phi}{\varphi}

\renewcommand{\rm}{\mathrm}

\newcommand{\m}{\mathbf}
\newcommand{\tendv}{\xrightarrow[n\rightarrow +\infty]{}}

\newcommand{\dr}{\partial}

\newcommand{\suml}{\sum\limits}

\newcommand{\id}{\mathrm{id}}

\newcommand{\bra}{\langle}
\newcommand{\ket}{\rangle}
\newcommand{\mc}{\mathcal}
\newcommand{\op}{\mathrm{op}}
\newcommand{\lk}{\mathrm{lk}\:}

\newcommand{\proj}{\mathrm{proj}}
\newcommand{\Conv}{\mathrm{Conv}}
\newcommand{\inv}{^{-1}}

\theoremstyle{definition}
\newtheorem{dfn}{Definition}[section]

\theoremstyle{plain}
\newtheorem{theo}[dfn]{Theorem} 
\newtheorem{pro}[dfn]{Proposition}
\newtheorem{lem}[dfn]{Lemma}
\newtheorem{cor}[dfn]{Corollary}
\newtheorem{prodfn}[dfn]{Proposition-Definition}

\newtheorem{theointro}{Theorem}

\newtheorem{corintro}[theointro]{Corollary}

\theoremstyle{remark}
\newtheorem{rem}[dfn]{Remark}
\newtheorem{ex}[dfn]{Example}

\title[Boundaries, buildings and rigidity]{Poisson boundaries of building lattices and rigidity with hyperbolic-like targets}
\author{Antoine Derimay}
\date{}
\begin{document}
\begin{abstract}
    We prove that a Poisson boundary of any regular thick Euclidean building, as well as lattices thereof is the space of chambers at infinity of the building with the harmonic measure class. We then use this result to generalize rigidity results of Guirardel-Horbez-Lécureux on morphisms and cocycles from lattices in buildings to groups with negative curvature properties. 
\end{abstract}
\maketitle

\tableofcontents
\section{Introduction}
Ever since the celebrated result of Mostow stating that hyperbolic manifolds of dimension at least $3$ are determined by their fundamental group \cite{Mos73}, followed by Margulis' superrigidity and normal subgroup theorems (see \cite{Zim84}), rigidity theory has been at the forefronts of geometric group theory. This work still continues to this day, for instance and most notably through the Zimmer program whose philosophy is that groups of high rank cannot act on manifolds of low dimensions. To this end, Zimmer's cocycle superrigidity theorem (see \cite{Zim84}), generalizing Margulis' superrigidity, is of crucial importance, highlighting the usefulness of cocycles in rigidity theory.

All the theorems mentioned above also make use of \emph{boundaries} in one way or another, which are geometric and/or measured objects encompassing the large-scale behaviour of the studied group. Originally studied by Furstenberg \cite{Fur63} in the context of semisimple Lie groups, boundary theory has been a central point of rigidity theory ever since, being studied (amongst others) by Margulis \cite{Mar91}, Zimmer \cite{Zim84}, Kaimanovich \cite{Kai00}, Bader-Shalom \cite{BS06}, and in more recent years by Bader and Furman \cite{BF14}. It is still a relevant object of study, providing a useful framework for proving rigidity results, as in \cite{GHL22} and \cite{BCFS24} for instance. However, boundaries are best understood in the setting of higher rank semisimple Lie groups, or higher rank semisimple algebraic groups over local fields, where most notions of boundary for this group turn out to be (closely related to) the flag variety $G/P$. 

One such notion of boundary is the Poisson-Furstenberg boundary. Given an irreducible aperiodic random walk $P$ on a countable set $X$, the \emph{Poisson-Furstenberg boundary} $B_{PF}(X,P)$ is a measured object which encompasses the large-scale behaviour of the walk, as well as the bounded harmonic functions via the \emph{Poisson formula} \cite{Woe00}. The case mostly studied in geometric/measured group theory is that of a group with a walk of the form $P(g,h)=\mu(g\inv h)$ for some measure $\mu$ on the group. The study of those Poisson boundaries has been a subject of its own right, with for instance results by Furstenberg (\cite{Fur73}, Kaimanovich-Vershik \cite{KV83} and Rosenblatt \cite{Ros81}) stating that a lcsc group $G$ is amenable if and only if there is some measure $\mu$ of full support for which $B_{PF}(G,\mu)$ is trivial. 

In some cases the Poisson boundary of the walk is somewhat easy to compute, and one might expect to retrieve informations about the Poisson boundary of walks on groups which act on the walk. We give here a tool that does just that, under the hypothesis that the group is big enough (see Theorem \ref{theo: inducedRW} and Proposition \ref{pro: integrability results of walk on lattice}).

\begin{theointro}\label{theointro: inducedRW}
    Let $(X,P)$ be an irreducible random walk on a countable discrete space, and $\Gamma<\Aut(X,P)$ be a discrete subgroup. If $\Gamma$ is a lattice in $X$, there is an admissible measure $\mu$ on $\Gamma$ such that
    $$B_{PF}(X,P)\simeq B_{PF}(\Gamma,\mu),$$
    where the isomorphism is $\Gamma$-equivariant.
    Moreover, $\mu$ is symmetric whenever $P$ is, and if $\Gamma$ is uniform then $\mu$ has a finite exponential moment.
\end{theointro}

This result is a simplified version of the famous Lyons-Sullivan procedure \cite{LS84} in the discrete setting. 

Here by \emph{lattice} in $X$ we mean a discrete subgroup of $\Aut(X)$ such that for some (equivalently any) fundamental domain $D$ for the action of $\Gamma$, 
$$\sum_{x\in D}\frac1{|\Gamma_x|}<\infty.$$
The proof of this proposition then comes from an explicit construction of a stationary distribution on $\Gamma\backslash X$.\medskip

Due to the founding work of Bruhat and Tits \cite{BT72}, \cite{BT84}, there is a deep link between reductive groups over local fields and Euclidean buildings. These are simplicial objects, which locally look like tilings of the Euclidean plane glued together. They have an associated natural CAT(0) metric, and are contractible. A Euclidean building of dimension $1$ is a tree without leaves, for which a lot can happen. In contrast, due to a result of Tits \cite{Tit86}, any irreducible Euclidean building in dimension at least $3$ actually comes from the Bruhat-Tits construction. In dimension $2$ however, the situation is mixed. There exist irreducible buildings which do not come from the Bruhat-Tits construction, but which still have cocompact automorphism groups \cite{Ron86}, \cite{Bar00}. These groups and their lattices still inherit many of the properties verified by higher rank semisimple groups over local fields \cite{KL97,KW14,LSW23,BFL23,Opp24}, showing that buildings are a natural generalization of algebraic groups. However, they differ from algebraic groups in a crucial point: they are not linear \cite{BCL19}, and they are even conjectured to not be residually finite, although this has for now been proven in only one example \cite{TMW23}, which is of type $\tilde C_2$. Alongside the normal subgroup theorem, which was proven in \cite{BFL23} only for type $\tilde A_2$, this would imply that any exotic building lattice is simple, providing numerous examples of finitely presented simple groups, a task that has been historically very hard to accomplish, and these groups would moreover enjoy higher rank rigidity properties of the classical case.

In this case too, the study of boundary actions is a key part of all the proofs of rigidity. In \cite{BF14}, following \cite{BM02}, Bader and Furman define a \emph{strong boundary} of a locally compact second countable group $\Gamma$ to be some measured space $(B,\nu)$ such that $\Gamma$ acts by measure-class preserving automorphisms, amenably and such that the diagonal action $\Gamma\curvearrowright B^2$ is isometrically ergodic. 

If $\Gamma$ is a lattice in some thick irreducible Euclidean building $X$ in the sense of Definition \ref{dfn: réseau dans un graphe}, one might expect that some boundary of $X$ would also be a boundary of $\Gamma$. Although a natural boundary to consider would be the Gromov boundary of the building, a geometric object which has a natural structure of spherical building, it turns out that the set $\Omega$ of its chambers corresponds to the Poisson-Furstenberg boundary in the classical case. It can be endowed with a family of natural measures, called harmonic measures, which are absolutely continuous with respect to each other. The $\Gamma$-action on $\Omega$ is then measure-class preserving, and we prove that it is a strong boundary of $\Gamma$. To do so, we exploit \cite[Theorem 2.7]{BF14}, telling us that any Poisson-Furstenberg boundary of a group is a strong boundary.\medskip

Using Theorem \ref{theointro: inducedRW}, we give the following description of one of the Poisson-Furstenberg boundaries of a building and of its lattices (see Theorems \ref{theo: Poisson boundary of building} and \ref{theo: discretization of walk}): 
\begin{theointro}\label{theo: intro BPF(X & Gamma)}
    Let $X$ be a regular thick Euclidean building, and let $\Gamma$ be a lattice in $X$. For any symmetric irreducible isotropic bounded range random walk $P$ on $X$, there exists a symmetric generating measure $\mu$ on $\Gamma$ with finite first moment such that 
    $$B_{PF}(\Gamma,\mu)=B_{PF}(X,P)=(\Omega,(\nu_x)_x),$$
    where $(\Omega,(\nu_x)_x)$ is the set of chambers at infinity of $X$, along with the harmonic measures. If moreover $\Gamma$ is uniform, $\mu$ has a finite exponential moment.
\end{theointro}
The first equality is an application of Theorem \ref{theointro: inducedRW}, which in this context resembles closely work done in \cite{Hae20}.

The second equality comes from a description of the Poisson-Furstenberg boundary for $(X,P)$ given in Theorem \ref{theo: intro BPF(X & Gamma)}. This description of the Poisson boundary was expected, with many earlier results pointing in this direction (\cite{Kai05,MZ03} for instance). We provide a proof here, building on the deep work of Rémy and Trojan \cite{RT21} describing the \emph{Martin boundary} of a building.

This discretization process was not known, even in the case of semisimple algebraic groups and their lattices acting on their Bruhat-Tits building. In that case, $(\Omega,\nu)$ corresponds to the flag variety $G/P$ for a minimal parabolic subgroup $P$ of $G$, with the Haar measure class. In this case, Theorem \ref{theo: intro BPF(X & Gamma)} can thus be written as follows:

\begin{corintro}
    Let $G$ be a product of semisimple algebraic groups over non-archimedean local fields. Let $\Gamma<G$ be a lattice, $P<G$ be a minimal parabolic subgroup, and $\nu$ be the Haar measure class on $G/P$. Then there exists a symmetric measure $\mu$ on $\Gamma$ such that 
    $$B_{PF}(\Gamma,\mu)=(G/P,\nu).$$
    Moreover, the measure $\mu$ always has finite first moment, and a finite exponential moment if $\Gamma$ is uniform.
\end{corintro}

\begin{rem}
    This partially answers a question raised in \cite[Remark 1.6]{BCFS24}, in that this gives a description of the Poisson-Furstenberg boundary of a lattice in a product of semisimple algebraic groups over local \emph{non archimedean} fields. The techniques developed here are however not enough to provide a description of the Poisson-Furstenberg boundary of a lattice in a product of semisimple algebraic groups over local fields, including $\m R$. This will be done in an upcoming work, using a more general discretization procedure.
\end{rem}

With this description of the Poisson-Furstenberg boundary, we prove a number of rigidity results using the work of Bader and Furman on boundary theory \cite{BF14}, and the formalism of \emph{geometrically rigid groups} developed in \cite{GHL22}. The notion of a group being geometrically rigid with respect to a discrete countable space has a convoluted definition (see Definition \ref{dfn: geometrically rigid}), but it is made to encompass all the properties a group with negative curvature properties should have, at least for their argument to work. Examples of geometrically rigid groups include hyperbolic groups and relatively hyperbolic groups, but also right-angled Artin groups, mapping-class groups and outer automorphism groups of torsion free hyperbolic groups. We get the following result:

\begin{theointro}\label{theo: intro morphisms}
    Let $\Gamma$ be a lattice in a product of regular higher rank Euclidean buildings. Let $\Lambda$ be a group which is geometrically rigid with respect to $\m D$. Then the image of any morphism $f:\Gamma\to\Lambda$ virtually fixes a point in $\m D$.
\end{theointro}

This theorem is Theorem \ref{theo: intro cocycles} with $X=\{\cdot\}$, an idea of the proof of Theorem \ref{theo: intro cocycles} is given there.\medskip

In many of the examples above, this can actually be strengthened to the fact that any morphism has finite image instead. See Corollary \ref{cor: grosse liste pour morphismes} for more details. 

\begin{rem}
    Note that in certain cases, this theorem was already known; we thus provide here a new proof of it. For instance, when $\Lambda=\Out(F_n)$ and $\Gamma$ is a lattice in an $\tilde A_2$ building, it follows from \cite{Zuk03} and \cite{BFL23} that the main result of \cite{BW11} can be applied to $\Gamma$, yielding that any morphism from $\Gamma$ to $\Lambda$ has to have finite image.

    If additionally $\Gamma$ is arithmetic, Haettel proved this in \cite{Hae20} for most of the groups listed here.
\end{rem}

Let $G,\Lambda$ be topological groups, and assume $G$ acts by measure-class preserving automorphisms on a standard space $(X,\mu)$. A \emph{cocycle} is a measurable map $c:G\times X\to \Lambda$, such that for all $g,h\in G$ and a.e. $x\in X$,
$$c(gh,x)=c(g,hx)c(h,x).$$
Two cocycles $c_1,c_2$ are \emph{cohomologous} if there exists a Borel map $f:X\to\Lambda$ such that
$$c_2(g,x)=f(gx)\inv c_1(g,x)f(x).$$
Cocycles are a generalization of a morphism that takes into account not just a group at the source, but also an action of that group on some measured space. Note that if $X=\{\cdot\}$, we recover the definition of a morphism, and of conjugacy of morphisms.

Zimmer was the first to use cocycles in rigidity theory \cite{Zim84}, his most famous result stating that any cocycle with source a higher rank semisimple algebraic group and target any semisimple algebraic group has to be cohomologous to a morphism. His results are the motivation behind Zimmer's program, stating that semisimple algebraic groups cannot act on manifolds of dimension smaller than their rank, culminating in \cite{BFH22}. They have also been a key argument in Furman's proof of the measure equivalence rigidity of lattices in higher rank simple Lie groups \cite{Fur99}, and many others. 

In \cite{GHL22}, the authors prove cocycle rigidity results when the source is still a semisimple algebraic group, but the target is changed to a group with hyperbolic properties, more precisely a geometrically rigid group. Here the idea is that higher rank lattices cannot act on spaces with hyperbolic properties. It is reflected in this context through the fact that any cocycle is cohomologous to a cocycle that takes its values in a smaller subgroup. We extend their result to all higher rank building lattices.

\begin{theointro}\label{theo: intro cocycles}
    Let $\Gamma$ be a lattice in a product of regular higher rank Euclidean buildings. Let $X$ be a standard probability space on which $\Gamma$ acts ergodically. Let $\Lambda$ be a group which is geometrically rigid with respect to $\m D$. Then any cocycle $c:\Gamma\times X\to\Lambda$ is cohomologous to a cocycle whose image virtually fixes a point in $\m D$.
\end{theointro}

Just like in the morphism case, for many of the known examples of geometrically rigid spaces, this can be strengthened to the fact that any cocycle is cohomologous to a cocycle whose image is finite (see Corollary \ref{cor: grosse liste pour cocycles}).

The proof of this Theorem requires two new ingredients, compared to \cite{GHL22}. The first one is proving that $(\Omega,\nu)$ is a strong boundary of $\Gamma$ in the sense of \cite{BF14}, which boils down to Theorem \ref{theo: intro BPF(X & Gamma)}. We then need to identify the \emph{special subgroups} of the \emph{generalized Weyl group} of the boundary $(\Omega,\nu)$ (see Section \ref{sec: Strong boundaries and Weyl groups} for definitions). In our case we cannot rely on the ambient algebraic group, so the proof of this fact is not trivial and requires new ideas. In fact we introduce new families of measures on $\Omega$ and $\Omega^2$ using the machinery of prouniform measures, introduced in \cite{BFL23}. This is done in Section \ref{sec: Strong boundaries and Weyl groups}.\medskip

\noindent\textbf{Organisation of the paper.} Section \ref{sec: Buildings} recalls various definitions from the theory of (Euclidean) buildings, as well as the theory of prouniform measures. Section \ref{sec: Random walks} is devoted to the definition of the various probabilistic objects we need, and to establishing Theorem \ref{theo: intro BPF(X & Gamma)}. In Section \ref{sec: Strong boundaries and Weyl groups} we introduce new families of measures on $\Omega$, designed to be uniform on the set of chambers at a certain Weyl distance, and use them to prove that the special subgroups of the Weyl group are the Weyl subgroups. We then use this and Theorem \ref{theo: intro BPF(X & Gamma)} in Section \ref{sec: Superrigidities} to prove Theorems \ref{theo: intro morphisms} and \ref{theo: intro cocycles}.

\noindent\textbf{Acknowledgements.} The author would like to thank Jean Lécureux and Camille Horbez for proposing this topic and answering a great many questions, Bertrand Rémy and Bartosz Trojan for taking the time to talk about the Poisson boundary of a building, Jean Lécureux and Stefan Witzel for fruitful discussions on prouniform measures in the type-preserving context and Malo Hillairet for instructive discussions regarding the probabilistic part of the paper.

\section{Buildings and measures}\label{sec: Buildings}
\subsection{Generalities about buildings}
\subsubsection{Coxeter groups and complexes}
For (a lot) more information on Coxeter groups, see \cite{Bou68} and \cite[Chapter III]{Bro89}.\medskip

A group $W$ is called a \emph{Coxeter group} if it admits a presentation 
$$G=\bra s\in S|(st)^{m_{s,t}}=1\ket$$
where $S$ is finite, $m_{s,s}=1$ and $m_{s,t}=m_{t,s}\in \m N_{\ge 2}\cup\{\infty\}$ when $s\ne t$. A Coxeter group is uniquely defined by its \emph{Coxeter matrix} 
$$M=(m_{s,t})_{s,t\in S},$$
or its \emph{Dynkin diagram} $D$, defined as follows. The vertices of $D$ are elements of $S$, and an edge is drawn between $s$ and $t$ if and only if $m_{s,t}>2$. The edge $\{s,t\}$ is labelled by $m_{s,t}$ when $m_{s,t}\ge 4$. The Coxeter group $W$ is said to be \emph{irreducible} if $D$ is connected. If $W$ is not irreducible, it is the direct product of its irreducible components, which correspond to the connected components of $D$. The \emph{length} $\ell(w)$ of $w\in W$ is the smallest integer $k$ for which there exist $s_1,\ldots s_k\in S$ (not necessarily distinct) for which $w=s_1\cdots s_k$.

From a Coxeter group $(W,S)$, one can form a simplicial complex, called the \emph{Coxeter complex}, defined as the poset 
$$\Sigma=\{w\bra S'\ket, w\in W,S'\subsetneq S\}$$ with the reverse inclusion order: $A\le B\iff A\supset B$. We will then say that $A$ is a face of $B$. Maximal simplices correspond to elements of $W$, and are called \emph{chambers}, or \emph{alcoves}. Simplices of codimension $1$ are called facets, while general simplices are faces.

The group $W$ acts on $\Sigma$ by left translation: 
$$w\cdot (w'\bra S'\ket)=ww'\bra S'\ket.$$
This action is simplicial, and simply transitive on the chambers. The \emph{type} (resp. \emph{cotype}) of a simplex $w\bra S'\ket$ will be $S\setminus S'$ (resp. $S'$). In particular, the type of a vertex is a singleton, and there are $n$ different vertex types. We will say of a simplicial automorphism of $\Sigma$ that it is \emph{type-preserving} if it preserves the type of every vertex, in which case it also preserves every type. The action of $W$ on $\Sigma$ is type preserving. A \emph{gallery} in $\Sigma$ is a finite sequence $(C_k)$ of chambers of $\Sigma$ for which $C_k$ and $C_{k+1}$ are \emph{adjacent}, in that they share a facet.

\subsubsection{Root systems and Euclidean reflection groups}
Let $\Phi$ be an irreducible root system on an $n$-dimensional Euclidean space $(E,\bra\cdot,\cdot\ket)$, and let $B=\{\alpha_1,\ldots ,\alpha_n\}$ be a basis of $\Phi$. Let $\Phi^+$ be the set of positive roots of $\Phi$ given by the choice of $B$. By irreducibility, there is a unique \emph{highest root} $\tilde\alpha\in \Phi$, i.e. a root for which all the coefficients $m_i$ in the expression $$\tilde\alpha=\sum_{i=1}^nm_i\alpha_i$$ are maximal among the elements of $\Phi$.

The dual basis $(\lambda_i)_i$ of the $\alpha_i$ is defined by $\bra \lambda_i,\alpha_j\ket =\delta_{i,j}$. The $\lambda_i$ are called the \emph{fundamental coweights} of $\Phi$, and $P=\bigoplus \m Z\lambda_i$ is called the \emph{coweight lattice}. A coweight $\lambda\in P$ is said to be \emph{dominant} if $\bra \lambda,\alpha_i\ket \ge 0$ for all $i$. We denote by $P^+$ the set of dominant coweights.

Denote by $\alpha^\vee=\frac{2\alpha}{\bra\alpha,\alpha\ket}$ the dual root of $\alpha$, and let $\Phi^\vee=\{\alpha^\vee,\alpha\in\Phi\}$ be the dual root system, and by $Q$ the \emph{coroot lattice} $\bigoplus \m Z\alpha_i^\vee$. It is the set of coweights $\lambda\in P$ with the same type as $0$.

We let $s_i:x\mapsto x-(\bra x,\alpha_i\ket)\alpha_i^\vee$ be the orthogonal reflection with respect to the hyperplane $\alpha_i^\perp$, as well as $s_0:x\mapsto x-(\bra x,\tilde\alpha\ket-1)\tilde\alpha^\vee$, the orthogonal reflection with respect to the hyperplane $\{\bra x,\tilde\alpha\ket =1\}$. Let $S_0=\{s_1,\ldots ,s_n\}$ and $S=\{s_0,\ldots ,s_n\}$, and let $W_0=\bra S_0\ket$ and $W=\bra S\ket$ be the subgroups of $O(E)$ (resp. $\Isom(E)$) generated by $S_0$ (resp. $S$). Both $(W_0,S_0)$ and $(W,S)$ are Coxeter groups, $W_0$ is finite, and $W$ is countable and discrete. The group $W_0$ is called the \emph{linear}, or \emph{spherical}, Weyl group, while $W$ is the \emph{affine} Weyl group. They will both be simply called Weyl group when it is clear from context which one is being referred.

A Coxeter group $(W,S)$ will be called \emph{Euclidean}, or \emph{affine}, if it arises as an affine Weyl group for some root system. In that case the Coxeter complex is isomorphic to the tessellation of $E$ by the hyperplanes 
$$\mc H=\{H_{\alpha,k}=\{x\in E,\bra x,\alpha\ket =k\},k\in \m Z,\alpha\in\Phi\}.$$
The vertex set of this tessellation is a lattice in $E$, of which $P$ is a sublattice. The set of types $i$ which can be elements of $P$ is denoted by $I_P$, it is the set of good types, as $P$ should be thought of as the set of good vertices of $\Sigma$. All vertices of these types will then be in $P$.

The fact that $\Sigma$ can be embedded in a Euclidean space in this way means we are able to make a choice of a preferred chamber in $\Sigma$ from our choice of a basis of $\Phi$. More specifically, the chamber $C_0=\{x\in E,\bra x,\alpha_i\ket >0, 1\le i\le n\text{ and }\bra x,\tilde\alpha\ket<1\}$ is called the \emph{fundamental chamber} of $\Sigma$.

In a similar way, the hyperplanes $\alpha_i^\perp$ separate $E$ into a finite number of connected components, which are open simplicial cones called \emph{sectors}. Among those, we can identify a \emph{fundamental sector}, defined by $\Lambda=\{x\in E,\bra x,\alpha_i\ket>0,1\le i\le n\}$ (note that $P^+=\overline\Lambda\cap P$). It is also a chamber at infinity, in the following sense: $E$ is a Euclidean space, hence $\rm{CAT}(0)$, so it has a Gromov boundary $\partial_\infty E\simeq \m S^{n-1}$. Then the sectors at infinity $\partial_\infty \mc S$ partition $\partial_\infty E$, and the Coxeter complex of $(W_0,S_0)$ is isomorphic to $\partial_\infty E$, where the chambers are the sectors at infinity $\partial_\infty\mc S$ (hence the name of \emph{spherical} Weyl group).

The \emph{extended} affine Weyl group is defined to be $\tilde W=W_0\rtimes P$, where $P$ acts on $E$ by translations. Note that $W$ is a normal subgroup of finite index in $\tilde W$. A simplicial automorphism $\psi$ of $\Sigma$ is said to be \emph{type-rotating} if it can be written as $w\circ \psi_0$ where $w\in\tilde W$ and $\psi_0$ is type-preserving.

\subsubsection{Buildings}
We are now ready to define the main object of our study: buildings. For more information about them, see \cite{AB08}. Let $X$ be a simplicial complex, let $(W,S)$ be a Coxeter group, and let $\Sigma$ be the associated Coxeter complex. We say that $X$ is a \emph{building of type} $W$ if $X$ is covered by a system of subcomplexes, called apartments, satisfying the following conditions:
\begin{enumerate}
	\item All the apartments are isomorphic to $\Sigma$,
	\item For any two simplices, there is a common apartment that contains them,
	\item If two apartments $\mc A_1,\mc A_2$ both contain simplices $A,B$, then there is an isomorphism between them which fixes $A$ and $B$ pointwise.
\end{enumerate}

We say that $X$ is affine (resp. spherical) when $(W,S)$ is. Many of the definitions we had for the Coxeter complex can be extended to buildings. For instance, there is a labelling of the vertices of $X$ into types which is compatible with that of $\Sigma$, in that the isomorphism between the apartment and $\Sigma$ can be chosen to be type-preserving. Two such labellings are essentially unique, up to a permutation of the types, which will have to be an automorphism of the Dynkin diagram of $(W,S)$. In particular, in the affine case, the set $V_P$ of \emph{good} vertices, those whose type is in $I_P$, does not depend on the choice of the labelling. The set $V_Q$ of vertices of type $0$ will however depend on the choice of labelling, albeit not in a dramatic way as we will always have $V_Q\subset V_P$. We make such a choice now and will not mention it later on.

Let $C,C'$ be chambers in $X$, and let $\mc A$ be an apartment containing $C,C'$. Since $W$ acts simply transitively on (the chambers of) $\Sigma$, there is a unique element $w\in W$ such that $w$ sends $C$ to $C'$ in $\mc A$. This $w$ can be seen not to depend on the choice of $\mc A$.

\begin{dfn}
    The element $w$ is called the $W$-distance between $C$ and $C'$, and is denoted by $\delta(C,C')$
\end{dfn}

This allows us to define the notion of a \emph{residue} (see \cite[Section 5.3]{AB08}). Let $J\subset S$, and let $W_J=\langle J\rangle$. Define an equivalence relation on the chambers of $X$ by $C\sim_J D$ if and only if $\delta(C,D)\in W_J$. The $\sim_J$-equivalence classes are called the $J$-residues of $X$. A subset $\mc R$ of $X$ is said to be a residue if it is a $J$-residue for some $J$, which is unique and called the \emph{type} of $\mc R$, while $|J|$ is the \emph{rank} of $\mc R$. $J$-residues are buildings of type $(W_J,J)$. For some residue $\mc R$ and some chamber $C\in X$, there is a unique chamber of $\mc R$ which is closest to $C$, it is called the \emph{projection} of $C$ onto $\mc R$, written $\proj_\mc R(C)$.

In the affine case, there is also a $P$-distance on $V_P$: for $x,y\in V_P$, choose an apartment containing both, and a type rotating isomorphism sending $x$ to $0$ and $y$ to some $\lambda\in P^+$ (this is always possible, and $\lambda$ is unique). 

\begin{dfn}
    Define the $P$ distance between $x$ and $y$ to be $\sigma(x,y)=\lambda$, and write $V_\lambda(x)$ to denote the (finite) set of points $y\in V_P$ such that $\sigma(x,y)=\lambda$.
\end{dfn} 

For a fixed $x$, the $V_\lambda(x),\lambda\in P^+$ partition $V_P$ (see \cite[Proposition 5.6]{1.pdf}). There it is also proven that $|V_\lambda(x)|$ does not depend on $x$; the common value is denoted $N_\lambda$.

Taking a facet in $X$, we may wonder how many chambers it belongs to. In $\Sigma$, that number is always $2$, and we shall regard $\Sigma$ as a not very interesting case of a building. We will instead prefer studying buildings where there is branching: we say a building is \emph{thick} if any facet belongs to at least $3$ chambers, or in other words if any residue of rank $1$ has at least three elements.

If $|\mc R|$ only depends on the type of $\mc R$ for  rank $1$ residues, we say the building is \emph{regular}.

\textbf{We will always assume our buildings to be affine, thick and regular from now on} unless explicitly stated otherwise, and we let $q_i+1$ be the common number among the facets of cotype $i$. If $w=s_{i_1}\cdots s_{i_k}$ is a reduced expression, we let $q_w=q_{i_1}\cdots q_{i_k}$. Also, for $\alpha\in \Phi$, we let $q_\alpha$ be equal to $q_i$ where $i$ is such that $\alpha\in W\cdot \alpha_i$.
This allows us to define a multiplicative function $\chi:E\to\m R_{>0}$ by
$$\chi(\lambda)=\prod_{\alpha\in\Phi^+}q_\alpha^{\bra \lambda,\alpha\ket}.$$
The function $\chi$ is useful when trying to compute $N_\lambda$. If for a finite subset $U\subset W$ we let
$$U(q\inv)=\sum_{w\in U} q_w\inv,$$
then 
$$N_\lambda=\frac{W_0(q\inv)}{(\Stab_{W_0}(\lambda))(q\inv)}\chi(\lambda)$$
(see \cite[Theorem 5.15]{1.pdf}). The space $X$ is also a metric space, with $d(x,y)$ being the distance of $x$ to $y$ for the euclidean distance in some apartment (as usual, this does not depend on the choice of apartment). $X$ endowed with this distance is $\rm{CAT}(0)$ (see \cite{AB08} for instance).

\subsection{The building at infinity}
Another definition from the Coxeter complex that is generalisable to buildings is that of sectors, and of points at infinity. A sector is simply a subcomplex of $X$ which is a sector in some apartment. Since $X$ is a $CAT(0)$ space, it admits a boundary at infinity $\dr_\infty X$, which we will sometimes denote $\Delta_\infty$. It is naturally endowed with a structure of a simplicial complex, with the maximal simplices being boundaries at infinity of sectors of $X$. Along with the observation made previously that the boundary of $\Sigma$ is isomorphic to the Coxeter complex of $(W_0,S_0)$, we find that the system of apartments made by the boundaries of apartments of $X$ turns $\Delta_\infty$ into a building of type $(W_0,S_0)$. For $x\in V_P$ and $\omega\in\Omega$, there is a unique sector $\mc S(x,\omega)$ which is based at $x$ and is in the class of $\omega$. The apartments of $\Delta_\infty$ are in one-to-one correspondence with apartments of $X$, and two sectors of  $X$ give the same chamber at infinity if and only if they contain a common subsector. The set of chambers at infinity will be the main focus of this paper, and will be denoted by $\Omega$.

Moreover, the building $\Delta_\infty$ being spherical, there is a unique element of longest length $w_0\in W_0$, and two chambers $\omega,\omega'\in\Omega$ are contained in a unique apartment if and only if $\delta(\omega,\omega')=w_0$. In this case, we say $\omega$ and $\omega'$ are \emph{opposite}, and thus there is a bijection between the set of apartments (with a choice of a fundamental sector), and the set $\Omega^{2,\op}$ of pairs of chambers at infinity which are opposite. We will be interested in the pairs of opposite chambers at infinity whose apartment contain a given vertex.

To this end, let $\lk o$ be the \emph{link} of $o$ in $X$, that is the set of facets $F$ not containing $o$ such that there is a chamber in $X$ which contains both $o$ and $F$, which we endow with the induced simplicial complex structure. Then for $o\in V_P$, $\lk o$ is a spherical building of type $(W_0,S_0)$, and for any $\omega\in\Omega,$ there is a unique element $\proj_{\lk o}(\omega)$ of $\lk o$ which is in $\mc S(o,\omega)$.

We have the following result:
\begin{lem}\label{lem: opp ds link ->o ds apt}
    Let $\omega,\omega'\in\Omega^{2,\op}$. Then $o\in V_P$ is in the apartment defined by $\omega,\omega'$ if and only if $\proj_{\lk o}(\omega)$, $\proj_{\lk o}(\omega')$ are opposite in $\lk o$.
\end{lem}

\begin{proof}
    If $o$ is in the apartment defined by $\omega,\omega'$ then there is an isometry from that apartment to $\Sigma$ that sends $o$ to $0$, $\mc S(o,\omega)$ to $\Lambda$ and $\mc S(o,\omega')$ to $w_0\Lambda$. But then, since $\proj_{\lk o}\omega=\lk o\cap\mc S(o,\omega)$ we must also have 
    $$\proj_{\lk 0}\partial_\infty(w_0\Lambda)=w_0\proj_{\lk 0}\partial_\infty\Lambda,$$
    so that $\proj_{\lk o}\omega',\proj_{\lk o}\omega$ are indeed opposite in $\lk o$.\medskip

    Conversely, if $\proj_{\lk o}(\omega),\proj_{\lk o}(\omega')$ are opposite, then $A=\mc S(o,\omega)\cup\mc S(o,\omega')$ is contained in an apartment by \cite[Proposition II.1.12]{Par00}. As there is exactly one apartment containing $\omega,\omega'$, we are done.
\end{proof}

Let $x,y\in V_P$ and $\omega\in\Omega$, and take some $z\in\mc S(x,\omega)\cap\mc S(y,\omega)$. The quantity
$$h(x,y;\omega):=\sigma(x,z)-\sigma(y,z)\in P$$
does not depend on $z$ (see \cite[Theorem 3.4]{2.pdf}) and is, in a sense, the Busemann function at $\omega$. The map $h$ additionally verifies the cocycle property: 
$$h(x,y;\omega)+h(y,z;\omega)=h(x,z;\omega)$$
for any $x,y,z\in V_P$, and $\omega\in\Omega$.

The set $\Omega$ can be endowed with a topology defined such that the sets 
$$\Omega_x(y)=\{\omega\in\Omega,\:y\in \mc S(x,\omega)\}$$
for some fixed $x\in V_P$ and varying $y\in V_P$ form a basis of clopen neighbourhoods. This topology makes $\Omega$ compact, and the topology does not depend on the choice of $x$ (see \cite[Theorem 3.17]{2.pdf}).

\subsection{Harmonic measures at infinity}\label{ssec:harm meas infty}
For a fixed $x\in V_P$, there is a unique Radon measure $\nu_x$ on $\Omega$ such that for any $y\in V_\lambda(x)$, 
$$\nu_x(\Omega_x(y))=\frac1{N_\lambda}$$
(see the discussion in \cite{2.pdf} following the proof of Theorem 3.6).

In a way, $\nu_x$ measures all directions at infinity with equal probability. As we will see later, this is showcased by the fact that $\nu_x$ captures the nearest neighbour random walk very well (or any isotropic walk in fact). The measures $\nu_x,x\in V_P$ are called the \emph{harmonic measures} on $\Omega$.

\begin{pro}\label{Radon-Nikodym nu}
    For any $x,y\in V_P$, the measures $\nu_x$ and $\nu_y$ are mutually absolutely continuous, and their Radon-Nikodym derivative is given by
    $$\frac{\mathrm d\nu_y}{\mathrm d\nu_x}(\omega)=\chi(h(x,y;\omega)).$$
\end{pro}

\begin{proof}
    See \cite[Theorem 3.17]{2.pdf}
\end{proof}

In particular, when we say that some subset $A\subset \Omega$ is of full $\nu_x$-measure, this notion will not depend on the specific $x$ we chose.

\begin{pro}
    Let $\omega\in\Omega$. The set $\Omega^{\op}(\omega)$ of chambers in $\Omega$ which are opposite to $\omega$ is of full $\nu_x$-measure.
\end{pro}

\begin{proof}
    See \cite[Theorem 6.2]{RT21}.
\end{proof}

\begin{cor}\label{cor:opposées de mesure pleine}
    The subset $\Omega^{2,\op}\subset \Omega^2$ of opposite chambers is of full $\nu_x\otimes \nu_x$-measure.
\end{cor}

\begin{proof}
    Since 
    $$\Omega^{2,\op}=\bigsqcup_{\omega\in\Omega}\{\omega\}\times\Omega^\op(\omega),$$
    we have
    $$\nu_x\otimes\nu_x(\Omega^{2,\op})=\int_\Omega \nu_x(\Omega^\op(\omega))\d \nu_x(\omega)=\int_\Omega 1\d\nu_x(\omega)=1.$$
\end{proof}

We now define a measure in the same measure class as $\nu_x$ that is $\Gamma$-invariant, generalizing \cite{BCL19} to all types of buildings.

\begin{pro}\label{dummy1}
    For $(\omega, \omega')\in\Omega^{2,\mathrm{op}}$ and $x\in V_P$, the quantity $$\beta_x(\omega,\omega'):=h(x,z;\omega)+h(x,z;\omega'),$$
with $z$ a point in the apartment containing $\omega$ and $\omega'$, does not depend on $z$. 

Furthermore, for every vertex $y\in V_P$, we have
$$\beta_x(\omega,\omega')-\beta_y(\omega,\omega') =h(x,y;\omega)+h(x,y;\omega').$$
\end{pro}

\begin{proof}
    This is Lemma 6.9 of \cite{BCL19}. More precisely, while \cite[Lemma 6.9]{BCL19} is stated and proven in the case of $\tilde A_2$ buildings, the proof actually works for a general affine building without modification.
\end{proof}

Now for a fixed vertex $x\in V_P$, we define the measure $m_x$ on $\Omega\times \Omega$ by the following formula: 
$$\mathrm dm_x(\omega,\omega')=\chi(-\beta_x(\omega,\omega'))\mathrm d\nu_x(\omega)\mathrm d\nu_x(\omega').$$
Note that $\beta_x$ is defined $\nu_x\otimes \nu_x$-a. e. by Proposition \ref{cor:opposées de mesure pleine}, which is enough for our purpose here.

Now if $x,y\in V_P$, by \ref{Radon-Nikodym nu} and \ref{dummy1}, we have
\begin{eqnarray*}
    \frac{\mathrm dm_x}{\mathrm dm_y}(\omega,\omega')=\chi(-(\beta_x(\omega,\omega')-\beta_y(\omega,\omega')))\frac{\mathrm d\nu_x}{\mathrm d\nu_y}(\omega)\frac{\mathrm d\nu_x}{\mathrm d\nu_y}(\omega')=1,
\end{eqnarray*}

so that $m_x=m_y$ for any $x,y$. \medskip

Thus the measure $m_x$ is independent of the choice of $x$. We get that this measure is also $\Aut(X)$-invariant, as for all $g\in\Aut(X)$ we have $g\nu_x=\nu_{gx}$ and $\beta_{gx}(g\omega,g\omega') = \beta_x(\omega,\omega')$, hence $gm_x = m_{gx} = m_x.$

We will call this common value $m$.

\begin{dfn}
    We let $\mathscr F$ be the space of simplicial embeddings of the Coxeter complex $\Sigma$ in the building $X$, equipped with the compact-open topology.
\end{dfn}

Following \cite{BFL23}, if we identify $\mathscr F$ with $\Omega^{2,\op}\times \Sigma,$ we may endow $\mathscr F$ with the measure $\zeta=m\otimes \lambda,$ where $\lambda$ is the counting measure on $\Sigma$. Besides being invariant under both actions on $\mathscr F,$ $\zeta$ is interesting because it is also a prouniform measure, which were first defined in \cite{BFL23}.

\subsection{Prouniform measures}\label{ssec: Prouniform measures}
This part is adapted from \cite[Section 1.1]{BFL23}, where prouniform measures were first defined. They prove to be the right framework in the proofs of Section \ref{sec: Strong boundaries and Weyl groups}. However, the fact that we are working with general buildings means that the $q_i$, the thickness of the panels of type $i$, might not all be equal, which complicates things. In particular, we will need to restrict ourselves only to \emph{type preserving} isometries, as predicted in \cite[Remark 3.21]{BFL23}. In other words, we assume here that all our simplicial complexes $X$ are additionally endowed with a \emph{type function} from the set of vertices of $X$ to a set of types, which we will always take to be $\{0,\ldots ,n\}$ for some $n$. We will then say for $X,Y$ simplicial complexes with a type function to the same type set that a map $f:X\to Y$ is type preserving if this type function is preserved by $f$.\medskip

For this part, fix a Euclidean Coxeter complex $\Sigma$, with the type function defined before. We will say that a finite simplicial complex $Y$ \emph{has} $P$ if it is a finite convex subset of $\Sigma$, in the sense that it is an intersection of half-apartments, and if it contains a vertex of type $0$.

\begin{dfn}\label{dfn: P-sym}
    Let $X$ be a typed simplicial complex. We say $X$ is $P$\emph{-symmetric} if for any connected finite non-empty simplicial complexes $Y,Y'$ having $P$, along with a type-preserving isometric embedding $i:Y'\to Y$, given a type-preserving isometric embedding $\alpha:Y'\to X$, the size of the set
    $$\{\beta:Y\to X,\:\beta \text{ is a type-preserving isometric embedding and }\alpha=\beta\circ i\}$$
    of type-preserving isometric embeddings of $Y$ into $X$ which extend $\alpha$ is positive and does not depend on the choice of $\alpha$.
\end{dfn}

From now on, unless otherwise specified, all our isometries and isometric embeddings of typed simplicial complexes will be type-preserving. Let $\Isom(Z,Z')$ be the set of (type preserving) isometries from $Z$ to $Z'$.

\begin{pro}\label{pro: buildings are symmetric}
    A regular building $X$ of type $\Sigma$ is $P$-symmetric.
\end{pro}

The proof of this fact is inspired from the proof in the dimension $2$ case in a paper in preparation by Stefan Witzel and Jean Lécureux.

\begin{proof}
    Let $Y\subset Z\subset \Sigma$ be finite convex simplicial complexes.     First off, since $Y$ is convex, the image of $Y$ by an embedding into $X$ will be contained in an apartment by \cite[Theorem 11.53]{AB08}. We may then choose one such apartment, and extend $Y$ to $Z$ in that apartment, so that the number of isometries extending any embedding of $Y$ to $Z$ is indeed positive.\medskip
    
    We now want to prove that for any embedding $\alpha:Y\to X$, the number of embeddings $\beta:Z\to X$ which extend $\alpha$ does not depend on the choice of $\alpha$. We start by reducing the problem to a simpler one.
    
    \begin{lem}\label{lem: pas de 1 dans la preuve de P-symmetricité}
        We may (and will) assume that $Z=\Conv(Y\cup\{z\})$ with $d(z,Y)=1$, where $d$ is the combinatorial distance in the $1$ skeleton of $X$.
    \end{lem}
    
    \begin{proof}
        Since $Z$ is finite, it is the convex hull of all of its points. Letting $Y_k$ be the convex hull of $Y$ and all the points in $Z^0$ at distance less than $k$ of $Y$, we have $Y_0=Y$, $Y_k=Z$ for $k$ large enough, and $d(z,Y_k)\le 1$ for any $z\in Y_{k+1}^0$ (note that we used convexity here). Thus by induction, we need only consider the case where $d(z,Y)\le 1$ for any $z\in Z^0$. Enumerate the points in $Z^0\setminus Y^0$ by $z_1,\ldots, z_k$, and let $Z_i$ be the convex hull of $Y$ and the points $z_1,\ldots ,z_i$. Again, $Z_0=Y$ and $Z_k=Z$. For a given $i<k$, either $Z_i=Z_{i+1}$ and there is nothing to do, or $Z_{i+1}=\Conv(Z_i,z_{i+1})$ and $d(z_{i+1},Z_i)>0$. In the latter case, $d(z_{i+1},Z_i)=1$, as it is an integer and $d(z_{i+1},Z_i)\le d(z_{i+1},Y)\le 1$. By induction again, we may assume that $Z=\Conv(Y\cup\{z\})$ and $d(z,Y)=1$, as wanted.
    \end{proof}

    \begin{lem}\label{lem: lk(z) cap Y is a simplex}
        The (non-empty) subcomplex $A:=\lk(z)\cap Y$ is a simplex.
    \end{lem}

    \begin{proof}
        Assume $A$ contains two points $a,b$ at distance at least $2$ from each other. Then the path $(a,z,b)$ is a path of length $2$ joining $a$ and $b$, so $z\in\Conv(A)$. Since $Y$ is convex and $A\subset Y$, we have $z\in\Conv(A)\subset Y$, a contradiction.
    \end{proof}

    \begin{lem}
        We may assume that $A$ is of codimension at most $1$ in $Y$.
    \end{lem}

    \begin{proof}
        Let $k$ be the codimension of $A$ in $Y$, and let $B=Z\cap\lk A$. Since $Z$ is convex, $B$ must be $\pi$-convex in $\lk A$, that is, any geodesic path in $\lk A$ joining two vertices of $B$ which are not opposite must be contained in $B$. 

        If $k\le 1$, we are done. Otherwise, $B\cap Y$ is of positive dimension, so it is not reduced to a point. There is then a point $b\in B\cap Y$ which is not opposite to $z$ ($\lk A$ is thin). Let $b=b_1,\ldots,b_n=z$ be a geodesic path in $\lk A$, which must thus be contained in $B\subset Z$. If we now let $Y_i$ be the convex hull of $Y$ and $b_1,\ldots,b_i$ by the same reasoning as Lemma \ref{lem: pas de 1 dans la preuve de P-symmetricité}, then either $Y_i=Y_{i+1}$ or $d_{\lk A}(b_{i+1},Y_i)=1$. By induction, we only need to treat the case where $z_1\in\lk A$ is such that $d_{\lk A}(z_1,Y)=1$. 

        By the exact same reasoning as Lemma \ref{lem: lk(z) cap Y is a simplex}, $A_1:=\lk_{\lk A} z_1\cap Y$ is again a simplex, of dimension strictly smaller than that of $A$. If $A_1$ is of codimension $1$ in $A$, then $\lk_\Sigma z_1\cap Y$ is of codimension $1$ in $Y$ as $\lk_\Sigma z_1\cap Y\supset\Conv(A_1,A)$. Otherwise, we repeat the process in $\lk_{\lk A}z_1$, which has a strictly smaller dimension than $\lk A$. This process must thus halt after a finite number of iterations, giving us a sequence $z_i$ for which $\lk_\Sigma z_{i+1}\cap Y_i$ is of codimension $1$ in $Y_i$, where $Y_i$ is the convex hull of $Y$ and $z_1,\ldots ,z_i$. This concludes the lemma.
    \end{proof}

    With this reduction in hand, we are now ready to move on to the actual proof of Proposition \ref{pro: buildings are symmetric}. We need to consider two cases, the first one being when $Y$ and $Z$ have different dimensions, i.e. $\dim A=\dim Y$. In that case, any point $x$ in the link of $\alpha(A)$ defines an isometry $\beta':Y\cup\{z\}\to X$ by $\beta'|_Y=\alpha$ and $\beta'(z)=x$. This extends to at most one isometry $\beta:Z\to X$, as $Z=\Conv(Y\cup\{z\})$. There always is one such isometry, as $\alpha(Y)$ is included in a half-apartment, and $x$ is included in a chamber which is next to that half-apartment, which means that $\alpha(Y)\cup\{x\}$ is included in an apartment by \cite[Theorem 11.53]{AB08}. But now the number of points in the link of $\alpha(A)$ does not depend on $\alpha$, only on the type of $A$ and the $q_i$'s.

    Assume now that the dimension of $Y$ and $Z$ are the same. Then $\dim A=\dim Y-1$. If $Z$ had $3$ or more simplices of maximal dimension containing $A$, we would have $\dim Z\ge \dim A+2$ since $Z$ is convex, $\Sigma$ is thin and $Z\ne Y$. Since $Y$ is convex, it has exactly one simplex of maximal dimension containing $A$. Let $y\in Y$ be such that $\Conv(A\cup\{y\})$ is that simplex. In order to define an extension of $\alpha$ to $Y\cup\{z\}$, we need only to determine the image of $z$. Let $x\in X$ be a potential candidate for the image of $z$. First of all, we need $x\ne\alpha(y)$. Now if $x$ and $\alpha(y)$ are not opposite in $\lk \alpha(A)$, then $\Conv(\alpha(A)\cup\{\alpha(y),x\})$ is of dimension at least $\dim A+2$, so $x$ cannot be the image of $z$ by any isometry extending $\alpha$. Conversely, if $x$ and $\alpha(y)$ are opposite in $\lk(\alpha(A))$, then $\beta'$ defined to be $\alpha$ on $Y$ and to send $z$ to $x$ is a type-preserving isometry from $Y\cup \{z\}$ to $\alpha(Y)\cup \{x\}$, since $y$ and $z$ are opposite in $\lk(A)$, it is also a metric isometry.
    This isometry extends to a unique isometry $\beta:Z\to X$ by the same argument as for the other case. Hence the number of isometries $\beta:Z\to X$ extending $\alpha$ is the number of points in $\lk \alpha(A)$ which are opposite to $\alpha (y)$, a quantity which again does not depend on $\alpha$ since $\alpha$ is type preserving, but only on the type of $A$ and $y$.

    Hence in either case, the number of extensions of $\alpha$ from $Y$ to $Z$ does not depend on $\alpha$, and only on $Y$ and $Z$, as wanted. 
\end{proof}

Fix now a regular building $X$ of type $\Sigma$, and an origin $o\in V_Q$. 

For $Y,Y'$ simplicial complexes satisfying $P$, and $i:Y'\to Y$ a (type-preserving) isometry, let $[Y,Y']_i$ be the number of extensions of some ${\alpha:Y'\to X}$ to $Y$, which does not depend on the choice of $\alpha$ by $P$-symmetricity of $X$. When $i$ is the inclusion, we will simply write $[Y,Y']$.

\begin{lem}
    Let $Y$ be a finite simplicial complex satisfying $P$.
    \begin{enumerate}
        \item If $i:Y\to Y$ is an isometry, $[Y,Y]_i=1$.
        \item For finite connected $P$-simplicial complexes $Y',Y''$ with isometries $i:Y''\to Y'$ and $j:Y'\to Y$, we have
        $$[Y,Y'']_{j\circ i}=[Y,Y']_j[Y',Y'']_i.$$
        \item For every $y,y'\in Y$ of the same type, 
        $$[Y,\{y\}]=[Y,\{y'\}].$$
    \end{enumerate}
\end{lem}

\begin{proof}
    The proof of the first two items is straightforward. For the third one, assume first that $y$ and $y'$ are at distance $2$ from each other, and let $E$ be the convex hull of $y$ and $y'$ in $\Sigma$. Note that there is an isometry $r$ of $E$ which swaps $y$ and $y'$. As such, 
    $$[E,\{y\}]=[E,E]_r[E,\{y\}]=[E,\{y'\}]$$
    by the first and second points of this proposition respectively. But now since $Y$ is convex, $E\subset Y$, so that by point 2 again, 
    $$[Y,\{y\}]=[Y,E][E,\{y\}]=[Y,E][E,\{y'\}]=[Y,\{y'\}].$$
    If $y$ and $y'$ are not assumed to be at distance $2$ anymore, up to taking $Y$ bigger, there is a sequence $y=y_0,y_1,\ldots,y_n=y'$ where all the $y_i$ have the same type and $y_i$ and $y_{i+1}$ are at distance $2$ from each other. Hence, in some (finite convex) $Z\supset Y$ we have 
    $$[Z,\{y\}]=[Z,\{y_1\}]=\cdots=[Z,\{y'\}].$$
    Finally, by (2), 
    $$[Z,Y][Y,\{y\}]=[Z,\{y\}]=[Z,\{y'\}]=[Z,Y][Y,\{y'\}],$$
    and finally $[Y,\{y\}]=[Y,\{y'\}]$ since $[Z,Y]\ne 0$.
\end{proof}

We may thus let $[Y]$ denote the common number $[Y,\{y\}]$ shared by all the vertices $y\in Y$ of type $0$. For $Y$ having $P$, let $m_Y$ be the counting measure on $\Isom(Y,X)$.\medskip

Fix an origin $y_0\in Y^0$ of type $0$ (the type of $o$). We let $\mu_Y^o$ be the uniform probability measure on the (finite) set $\Isom(Y,X)^o$ of embeddings $\alpha$ such that $\alpha(y_0)=o.$ More generally, if $i:Y'\to Y$ and $\alpha:Y'\to X$ are isometries, we denote by $\mu_{Y,Y'}^\alpha$ the uniform measure on the set
$$\Isom(Y,X)^\alpha:=\{\beta\in\Isom(Y,X)\:|\beta\circ i=\alpha\}.$$
Note that if $Y'$ is the vertex (of type $0$), we recover the previous definition.

We also let $\mu_Y$ be the measure $\sum_o\mu_Y^o$, where the sum ranges over all $o\in X^0$ with the same type as $y_0$. This definition clearly does not depend on the choice of the basepoint in $Y$ (of type $0$).\medskip

From an embedding $i:Y'\to Y,$ we can induce the restriction map $i^*:\Isom(Y,X)\to \Isom(Y',X)$ to be the precomposition by $i$.

\begin{pro}\:
    \begin{itemize}
        \item We have $[Y]\mu_Y=m_Y$.
        \item The restriction maps are measure preserving, i.e. for any non-empty connected simplicial complexes $Y,Y'$ having $P$, and any embedding $i:Y'\to Y$,
    $$(i^*)_*\mu_{Y}=\mu_{Y'}.$$
    \end{itemize}
\end{pro}

\begin{proof}
    This is Lemma 1.4 of \cite{BFL23}, the proof of both points carries over immediately.
\end{proof}

The next statement as well as its proof, is the same as that of \cite[Lemma 1.5]{BFL23}, it is a double counting argument.

\begin{lem}\label{lem: disintegration of prouniform in finite case}
    For $Y_0,Y_1,Y_2$ finite simplicial complexes having $P$, with isometries $i_k:Y_k\to Y_{k+1}$ and $\alpha_0\in\Isom(Y_0,X)$. Then 
    $$\mu_{Y_2,Y_0}^{\alpha_0}=\sum_{\alpha_1\in \Isom(Y_1,X)}\mu_{Y_1,Y_0}^{\alpha_0}(\alpha_1)\mu_{Y_2,Y_1}^{\alpha_1}.$$
\end{lem}

Now that we have a family of measured spaces and measure preserving maps between them, we can take the limit of those, in order to also have measures on $\Isom(Y,X)$ for infinite $Y$.\medskip

\begin{dfn}
    We say that a simplicial complex $Y$ is $\mathrm{ind}$-$P$ if it can be written as an ascending union of finite convex subcomplexes having property $P$.
\end{dfn}

Note that if $Y$ is ind-$P$, then 
$$\Isom(Y,X)=\varprojlim\Isom(Y',X)$$
in the category of sets, where the limit ranges over all finite convex subcomplexes $Y'$ having $P$, and the transition maps between two subcomplexes being the restriction maps.

Since the restriction maps are measure preserving the projective limit taken above can actually be considered in the category of measured spaces, so we can define a measure $\mu_Y$ on $\Isom(Y,X)$ to be
$$(\Isom(Y,X),\mu_Y)=\varprojlim(\Isom(Y',X),\mu_{Y'}),$$
that is, the only measure such that for any finite convex subcomplex $Y'\subset Y$ having $P$, if $i:Y'\to Y$ designates the inclusion, then $(i^*)_*\mu_Y=\mu_{Y'}.$

\begin{dfn}
    The measure $\mu_Y$ is called the \emph{prouniform measure} on $\Isom(Y,X)$. 
\end{dfn}

The extension of the previous results from $P$ to ind-$P$ complexes works now in the exact same way as in \cite{BFL23}. We don't include the proofs of the following results, but we write them down in order to be able to reference them later on in the text.

More generally, for a pair of connected ind-$P$ simplicial complexes $W,Y$ and an isometry $i:W \to Y$, and $\alpha\in \Isom(W,X)$, we can construct a relative version $\mu_{Y,W}^\alpha$ as follows.
Let $W'$ be a connected  ind-$P$ simplicial complex, with isometries $i':W\to W'$, $j:W'\to Y$ such that $i=j\circ i'$. Assume that $i'(W)$ contains all but finitely many simplices of $W'$.

The following lemma is \cite[Lemma 1.8]{BFL23}.
\begin{lem}
The set $\Isom(W',X)^\alpha=\{\beta\in\Isom(W',X)\mid \alpha=\beta\circ i'\}$ is finite, and its size is independent of $\alpha$.
\end{lem}

Thus, for $W,W'$ as above, we can equip $\Isom(W',X)^\alpha$ with the uniform measure, which we denote $\mu_{W',X}^\alpha$.

More generally, define $\Isom(Y,X)^\alpha = \{\beta\in \Isom(Y,X)\mid \alpha=\beta\circ i\}$. Observe that again, in the category of sets, we have
$$ \Isom(Y,X)^\alpha= \varprojlim \Isom(W',X)^\alpha, $$
where $W'$ runs over all cofinite complexes containing $W$ as a subcomplex and being ind-$P$. 
We can thus introduce a measure $\mu_{Y,W}^\alpha $ on $\Isom(Y,X)^\alpha$ by setting
$$(\Isom(Y,X)^\alpha,\mu_{Y,W}^\alpha) = \varprojlim (\Isom(W',X),\mu_{Y,W'}^\alpha).$$

In particular, when $W$ is reduced to a point, we define this way the \emph{restricted} measure $\mu_Y^o$ for every $o\in X_0$. Note that again we get
$$\mu_Y=\sum_{o\in X_0} \mu_Y^o$$

\begin{dfn} 
The measures $\mu_Y^\alpha$ are called \emph{restricted prouniform measures}.
\end{dfn}

The following is \cite[Lemma 1.10]{BFL23}, it consists in taking the limit in Lemma \ref{lem: disintegration of prouniform in finite case}.
\begin{lem}\label{lem: desintegration des prouniformes}
    Let $Y_0,Y_1,Y_2$ be ind-$P$ simplicial complexes, with embeddings $i_k:Y_k\to Y_{k+1}$. Let $\alpha_0:Y_0\to X$ be an embedding. Then
    $$\mu_{Y_2,Y_0}^{\alpha_0} = \int_{\Isom(Y_1,X)} \mu_{Y_2,Y_1}^{\alpha_1} d\mu_{Y_1,Y_0}^{\alpha_0} (\alpha_1)$$
\end{lem}

As a direct consequence of the Lemma, we get
\begin{cor}\label{cor: i star star}
Let $Y_0,Y_1,Y_2$ be connected ind-$P$ simplicial complexes, with embeddings $i_k:Y_k\to Y_{k+1}$, and $\alpha_0:Y_0\to X$ be an embedding. 

Then $(i_1^*)_*\mu_{Y_2,Y_0}^{\alpha_0} = \mu_{Y_1,Y_0}^{\alpha_0}.$
\end{cor}

\begin{ex}
    Let $\Lambda$ be the fundamental sector of $\Sigma$. For $o\in V_Q$, we let $\mu_\Lambda^o$ be the restricted prouniform measure on $\Isom(\Lambda,X)^o$, with the restriction that $\phi(0)=o$. Then $\Isom(\Lambda,X)^o$ can be identified with $\Omega$. Since the harmonic measure $\nu_o$ satisfies all the characteristics of the prouniform measure, we have 
    $$\mu_\Lambda^o=\nu_o,$$
    linking this definition with our previously defined harmonic measures.
\end{ex}

We may also endow $\mathscr F=\Isom(\Sigma,X)$ with the prouniform measure $\mu_\mathscr F$.

Recall that by identifying $\mathscr F$ with $\Omega^{2,\op}\times \Sigma,$ we had defined a measure $\zeta_\mathscr F:=m\otimes\lambda$ on $\mathscr F.$

The proof of the following Lemma works the same as that of Lemma 3.23 of \cite{BFL23}.
\begin{lem}
    The measures $\mu_\mathscr F$ and $\zeta_\mathscr F$ are equal.
\end{lem}

\begin{cor}\label{Weyl preserve la mesure}
    The measures $\mu_\mathscr F$ and $m$ are invariant under the action of the (spherical) Weyl group.
\end{cor}

\begin{proof}
    Since $\Sigma$ is $W$-invariant and $W$ is acting simplicially, we directly have $w\mu_\mathscr F=\mu_\mathscr F$.

    Now under the identification $\mathscr F\simeq \Omega^{2,\op}\times \Sigma$, $W$ stabilizes $\Omega^{2,\op}\times\{0\},$ and the restriction of $\mu_\mathscr F$ is precisely $m$, which is thus also $W$-invariant.
\end{proof}

\section{Random walks}\label{sec: Random walks}
\subsection{Graphs and Markov chains}\label{ssec:defs random walks}
We start by recalling basic facts on random walks, following for the most parts the book \cite{Woe00}.

A graph $X=(V,E)$ will be the data of a vertex set $V$ and a symmetric $E\subset V^2$ (note that we allow edges from $x$ to itself). We will often forget $V$, and simply refer to vertices as elements of the graph. We will often write $x\sim y$ instead of $(x,y)\in E$. Note that $\sim$ is symmetric, but need not be reflexive nor transitive.

A graph is naturally a metric space, by letting 
$$d(x,y)=\inf\{n\in\m N,\exists z_0,z_1,\ldots,z_n\in X,\:x=z_0\sim z_1\sim\cdots\sim z_n=y\}$$
be the shortest path distance. We say a graph is \emph{connected} if $d$ is always finite.

The degree of $x$ is the number $\deg(x)$ of $y\in X$ such that $x\sim y$, and $\lk(x)$ is the set of such $y$. We say $X$ is \emph{locally finite} if $\deg(x)<\infty$ for all $x\in X$.\\

A \emph{Markov chain} on a countable discrete set $X$ is the data of a \emph{transition matrix} 
$$P=(p(x,y))_{x,y\in X}$$
such that $p(x,y)\ge 0$, and
$$\sum_{y\in X}p(x,y)=1$$
for all $x\in X$. One can think of $P$ as encoding the probabilities an ant might go from a vertex to another (in one unit of time), so that the probability of going to $y$ when at $x$ is $p(x,y)$. We then get a sequence of random variables $(Z_n)_{n\in\m N}$ on $X$, representing the position we are at at time $n$ when starting at a given point. The probability space on which the $Z_n$ live is the trajectory space $X^{\m N}$, so that $Z_n$ is simply the $n$-th coordinate. If we start the walk at $x\in X$, the probability on $X^\m N$ is the only one verifying 
$$\m P_x(Z_0=x_0,Z_1=x_1,\ldots,Z_k=x_k)=\m 1_x(x_0)p(x_0,x_1)\cdots p(x_{k-1},x_k)$$
for all $k\in\m N$ and $x_i\in X$. The associated expectation will be denoted $\m E_x$.

Let $p^{(n)}(x,y):=\m P_x(Z_n=y)$ be the probability that we are at $y$ at time $n$ when we start at $x$. We say $P$ is \emph{irreducible} if for any $x,y\in X$, there is some $n$ for which $p^{(n)}(x,y)>0$. In other words, we can go from $x$ to $y$ with positive probability. We also say $P$ is \emph{symmetric} if $p(x,y)=p(y,x)$. In that case, we also have $p^{(n)}(x,y)=p^{(n)}(y,x)$.

Finally, we say that $P$ is \emph{recurrent} if 
$$\m P_x(\exists n>0, Z_n=x)=1,$$
that is we come back to where we started with probability $1$. Although the definition may depend on $x$, it does not when the walk is irreducible, which we assume from now on. We say that $P$ is \emph{transient} if it is not recurrent.

Define $\Aut(X,P)$ as the group of bijections of $X$ which preserve $P$, that is, the group of $f\in\rm{Sym}(X)$ such that $P(f(x),f(y))=P(x,y)$ for all $x,y\in X$.

The transition kernel $P$ can also be seen as a continuous operator, called an \emph{averaging operator} on $\m R^X$, by letting
$$Pf(x)=\m E_x(f(Z_1))=\sum_{y\in X}p(x,y)f(y).$$
(actually on the subspace of functions $f:X\to\m R$ such that $P|f|<\infty$).\\
A function $f$ is said to be $P$-\emph{harmonic} if $f=Pf$. In that case, $(f(Z_n))_n$ is a martingale with respect to the filtration generated by the $(Z_n)$.\\

Dually, say a measure $\nu$ on $X$ is $P$\emph{-stationary} if 
$$\sum \nu(x)p(x,y)=\nu(y),$$
which is often written as $\nu P=\nu$.\\

When given a locally finite graph, we can form a Markov chain on it by saying that we are equally likely to take any from $x$, i.e.
$$p(x,y)=\frac1{\deg(x)} \text{ if } x\sim y,\text{ and $0$ otherwise.}$$
This Markov chain is the one most commonly used, and is called the \emph{simple random walk} on $X$. It is irreducible precisely when the graph is connected.

Conversely, if $(X,P)$ is a Markov chain, if we have the added assumption that $p(x,y)>0\iff p(y,x)>0$, we can define a graph structure on $X$ by saying that $x\sim y$ exactly when $p(x,y)>0$.\medskip

Random walks also appear naturally in the context of group theory. Let $G$ be a locally compact group, and let $\mu$ be a probability measure on $X$, which we assume to be absolutely continuous with respect to the Haar measure on $G$. Assume $\mu$ is admissible (the subgroup generated by $\mathrm{supp}(\mu)$ is $G$). We then let $p(g,h)=\mu(g\inv h)$. The above hypotheses ensure that our walk is irreducible, and it will be symmetric whenever $\mu$ is.

\subsection{Poisson-Furstenberg boundaries}
The Poisson-Furstenberg boundary is a measured space associated to a random walk. The idea is that this space encompasses all the trajectories at infinity of the random walk, along with the hitting probability. The reader is referred to \cite{Kai96} or \cite{Lal23} for more details.\medskip

Let $(X,P)$ be an irreducible random walk on a discrete space, and fix an origin $o\in X$. Recall that the path space $(X^\m N,\m P_o)$ is the space $X^\m N$ of sequences of points in $X$, and $\m P_o$ was defined in the last section.

Say two sequences $(x_n)$ and $(x_n')$ in the path space are $\sim$-equivalent if there are integers $m,m'$ such that $T^m(x_n)=T^{m'}(x_n'),$ where $T((x_n))=(x_{n+1})$ is the shift. In other words, $(x_n)$ and $(x_{n+1})$ have the same tail.

Let $\mc A$ be the $\sigma$-algebra of all $\sim$-closed measurable subsets of $X^\m N$. The Poisson-Furstenberg boundary of $(X,P)$ is the unique standard probability space $B_{PF}(X,P)$ (up to isomorphism) such that there is a measurable, surjective map 
$$\pi:X^\m N\twoheadrightarrow B_{PF}(X,P)$$ 
such that $\mc A$ is the pullback of the measurable subsets of $B_{PF}(X,P)$ by $\pi$. We let $\nu_x=\pi_*\m P_x$, and since $P$ is irreducible, the measures $\nu_x$ are all absolutely continuous with respect to one another. They are moreover stationnary, in the sense that 
$$\sum_{y\in X}p(x,y)\nu_y=\nu_x.$$

In the case where $X$ is a countable group $\Gamma$ and $P$ is defined in the same way as described in \ref{ssec:defs random walks}, the tail equivalence $\sim$ is clearly $\Gamma$-equivariant, so that the (diagonal) action of $\Gamma$ factors to one on $B_{PF}(\Gamma,\mu)$. The theory is very developed in that specific context, and the Poisson boundary often gives actions of $\Gamma$ with nice properties. For instance, the action is ergodic when acting diagonally on $B_{PF}(\Gamma,\mu)^2$, and also amenable. This means that Poisson-Furstenberg boundaries are \emph{boundaries} of $\Gamma$ in the sense of Bader-Furman (see Example \ref{ex: BPF is strong boundary}). The reader is referred to \cite{Kai00} for more details as well as examples on the Poisson-Furstenberg boundary of a group.\medskip

A big interest of Poisson-Furstenberg boundaries is its relationship with harmonic functions, via the \emph{Poisson formula}.

For a bounded harmonic function $h$ on $X$, $(h(Z_n))_{n\in\m N}$ is a bounded martingale, so that Doob's martingale convergence Theorem tells us that for $\m P_o$-almost every path $(x_n)_{n\in\m N}$, $(h(x_n))_n$ converges to some $f(\bar x)$. This limit is clearly $\sim$-invariant, so that $f$ descends to a map in $L^\infty(B_{PF}(X,\mu))$, which we keep calling $f$. We may recover $h$ from $f$ thanks to the Poisson integral formula, which gives
$$h(x)=\int_{B_{PF}(X,P)}f(\omega)\d \nu_x(\omega).$$
This formula can be applied to any $f\in L^\infty(B_{PF}(X,\mu))$, and the obtained map $h_f$ is a bounded harmonic function on $X$. The two maps we just defined are inverses of each other, and actually are ($\Aut(X,P)$-equivariant) isomorphisms between the Banach spaces $H^\infty(X,P)$ and $L^\infty(B_{PF}(X,\mu))$.

In particular, in order to prove that two Poisson boundaries are isomorphic, it suffices to prove that their spaces of bounded harmonic functions are.

\subsection{The induced walk on a recurrent subset}\label{ssec: Induced walk}
In this part we describe how to transform a random walk into another walk on a big enough subset of $X$. This method is heavily inspired by the work of Lyons and Sullivan (see \cite[Section 7]{LS84}), and the results here should in fact be thought of as a discrete and thus simplified analogue of (a part of) their work.

\begin{dfn}\label{dfn: recurrent subset}
    Let $X$ be a countable discrete space, and let $(Z_n)$ be a random walk on $X$ with transition matrix $P$. We say that a subset $Y\subset X$ is \emph{recurrent} if, for any $x\in X$,
    $$\m P_x(\exists n,Z_n\in Y)=1.$$
    We say that $Y$ is transient if it is not recurrent.

    For a recurrent $Y\subset X$, define inductively a family of random variables $\tau_k$ by $\tau_{-1}=-1,$ and 
    $$\tau_{k+1}=\inf\{j>\tau_k,Z_j\in Y\}.$$
\end{dfn}

Since $Y$ is recurrent, the $\tau_k$ are a.s. finite stopping times for the filtration generated by the $(Z_n)$.

This section is devoted to the proof of the following:

\begin{theo}\label{theo: inducedRW}
    Let $(X,P)$ be a random walk, with associated random variables $(Z_n)_n$, and let $Y\subset X$ be recurrent. Define the random walk $Q$ on $Y$ by
    $$q(x,y)=\m P_x(Z_{\tau_1}=y),$$
    or equivalently by the $Y$-valued random variables $(S_k)_{k\in\m N}:=(Z_{\tau_k})_{k\in\m N}$.
    
    Then there is an isomorphism of Banach spaces
    \begin{eqnarray*}
        H^\infty(X,P)&\simeq& H^\infty(Y,Q)\\
        h&\mapsto& h|_Y\\
        \m E_x(h(S_0))&\mapsfrom&h
    \end{eqnarray*}
    between bounded harmonic functions on $X$ and bounded harmonic functions on $Y$. Moreover, 
    \begin{enumerate}
        \item If $P$ is symmetric, then so is $Q$.
        \item If $\Gamma<\Aut(X,P)$ is such that $Y$ is $\Gamma$-invariant, then so is $Q$, and the above isomorphism is $\Gamma$-equivariant.
    \end{enumerate}
\end{theo}

\begin{proof}
    Take a bounded harmonic function $h$ on $X$, we aim to prove that its restriction $h|_Y$ to $Y$ is again harmonic.
    
    By harmonicity of $h$, the sequence $h(Z_n)$ is a martingale with respect to the filtration generated by the $Z_i$. Since $h$ is bounded, the $h(Z_n)$ are uniformly integrable, so by the optional stopping Theorem (\cite[Theorem 12.28]{LG22}),
    $$\m E_x(h(S_1))=\m E_x(h(Z_{\tau_1}))=\m E_x(h(Z_0))=h(x)$$
    for any $x\in Y$, which is precisely saying that $h|_Y$ is harmonic for the induced random walk on $Y$.\medskip
    
    Conversely, let $h$ be a harmonic function on $Y$ and let
    $$\tilde h(x)=\m E_x(h(S_0))=\sum_{y\in Y}\alpha(x,y)h(y),$$
    where $\alpha(x,y)=\m P_x(S_0=y)$

    Note that for $x\in Y$, $\alpha(x,y)=\delta_{x,y},$ so that $h$ and $\tilde h$ coincide on $Y$.
    
    Since $\suml_{y\in Y}\alpha(x,y)=1,$ $\tilde h$ is bounded by the same constant as $h$, and it only remains to prove harmonicity of $\tilde h$.
    
    We have 
    \begin{eqnarray*}
        \sum_{y\in X}p(x,y)\tilde h(y)&=&\sum_{y\in X}p(x,y)\sum_{z\in Y}\alpha(y,z)h(z)\\
        &=&\sum_{z\in Y}h(z)\sum_{y\in X}p(x,y)\alpha(y,z)
    \end{eqnarray*}
    Now letting $\sigma=\inf\{j\ge 1,Z_j\in Y\},$ we have 
    \begin{eqnarray*}
        \m P_x(Z_\sigma=z)&=&\sum_{y\in Y}\m P_x(Z_1=y)\m P_x(Z_\sigma=z|Z_1=y)\\
        &=&\sum_{y\in Y}p(x,y)\m P_y(Z_{\tau_0}=z)\\
        &=&\suml_{y\in X}p(x,y)\alpha(y,z).
    \end{eqnarray*}
    
    Note that $\sigma=\tau_0\m 1_{\{Z_0\not\in Y\}}+\tau_1\m 1_{\{Z_0\in Y\}}.$ Hence if $x\not\in Y$, 
    $$\sum_{y\in X}p(x,y)\alpha(y,z)=\m P_x(S_0=z)=\alpha(x,z),$$
    so that 
    $$\sum_{y\in X}p(x,y)\tilde h(y)=\sum_{z\in Y}h(z)\sum_{y\in X}p(x,y)\alpha(y,z)=\sum_{z\in Y}\alpha(x,z)h(z)=\tilde h(x).$$
    Now if $x\in Y,$ we get 
    $$\sum_{y\in X}p(x,y)\alpha(y,z)=\m P_x(S_1=z)=q(x,z)$$
    instead, and
    $$\sum_{y\in X}p(x,y)\tilde h(y)=\sum_{z\in Y}h(z)\sum_{y\in X}p(x,y)\alpha(y,z)=\sum_{z\in Y}q(x,z)h(z)\underset{(*)}=h(x)=\tilde h(x),$$
    where $(*)$ is the $Q$-harmonicity of $h$.
    
    In the end, $\tilde h$ is harmonic on $X$, and we already proved that $(\tilde h)|_Y=h.$ Finally, if $h$ is bounded harmonic on $X$, then 
    $$\widetilde{h|_Y}(x)=\m E_x(h|_Y(S_0))=\m E_x(h(S_0))=\m E_x(h(Z_{\tau_0}))=\m E_x(h(Z_0))=h(x)$$
    again by Doob's stopping time Theorem, and we are done.

    The moreover part is straightforward and left to the reader.
\end{proof}

\subsection{A recurrence criterion}\label{ssec: networks}
Given Theorem \ref{theo: inducedRW}, it makes sense to seek out a criterion for recurrence of certain subsets.

In order to state our results in a greater level of generality, we give the definition of a \emph{lattice} in a general graph-theoretic setting (see \cite{BL01}).

\begin{dfn}\label{dfn: réseau dans un graphe}
    Let $X$ be a connected, locally finite graph. A lattice in $X$ is a subgroup $\Gamma<\Aut(X)$ such that all stabilizers are finite, and for some fundamental domain $D$ for the action of $\Gamma$ on $X$,
    $$\sum_{x\in D}\frac1{|\Gamma_x|}<\infty.$$
\end{dfn}
When $\Aut(X)$ is acting cocompactly on $X$, the group $\Gamma$ is a lattice in $X$ exactly when it is a lattice in the locally compact group $\Aut(X)$, hence the name.

Our strategy for proving that a certain walk is recurrent will be to exhibit a finite invariant measure. Although this could be done without the definition of \emph{networks}, they provide the natural intuition as to why this works.

\begin{dfn}
     A \emph{network} is the data of a countably infinite set $X$, along with a symmetric map $a:X\times X\to \m R_{\ge 0}$ such that for any $x\in X$, $m(x):=\sum_ya(x,y)$ is positive and finite.

    The map $a$ is called the \emph{conductance}, and $m$ is the \emph{total conductance}. Letting $p(x,y)=\frac{a(x,y)}{m(x)}$ defines a Markov chain for which $m$ is stationary. Conversely, a Markov chain is called \emph{reversible} if it can be obtained in this way. 

    We will say that a network is \emph{recurrent}, \emph{transient} or \emph{irreducible} if the associated Markov chain is.
\end{dfn}

For instance, the simple random walk on a locally finite graph is always reversible, with $a(x,y)=1$ if $x\sim y$ and $0$ otherwise.\medskip

We are now able to state the result of this section.

\begin{pro}\label{pro: (positive) recurrence criterion}
    Let $\mc N=(X,a)$ be a network, $P$ the associated random walk, and let $\Gamma<\Aut(X,P)$ be such that $|\Gamma_x|$ is finite for all $x$, and such that for some (equivalently, any) fundamental domain $D$ of $\Gamma\backslash X$,
    $$\sum_{x\in D}\frac{m(x)}{|\Gamma_x|}<\infty.$$ 
    Then for any $o\in X$, the random walk on $X$ meets $\Gamma\cdot o$ infinitely often almost surely.
\end{pro}

As a direct corollary, we get:

\begin{cor}\label{cor: lattice -> recurrence of orbits}
    Let $X$ be a connected, locally finite graph with bounded degree, and let ${\Gamma<\Aut(X)}$ be a lattice in $X$. Then for any base point $o\in X$, the simple random walk on $X$ (or any {$\Gamma$-invariant} network on $X$ with bounded total conductance) meets $\Gamma\cdot o$ infinitely often almost surely.
\end{cor}

In order to prove the theorem, we will prove recurrence of the \emph{quotient network}.
\begin{prodfn}\label{dfn: réseau quotient}
    Let $\mc N=(X,a)$ be an irreducible network, and let $\Gamma<\Aut(\mc N)$ (i.e. $a(\gamma x,\gamma y)=a(x,y)$ for all $x,y\in X$ and $\gamma\in\Gamma$) be a closed subgroup.
    
    Then the data of a conductance $a'$ on $\Gamma\backslash X$ defined by
    $$a'(x,y)=\frac1{|\Gamma_x|}\sum_{y'\in\pi\inv(y)}a(x',y'),$$
    where $\pi$ is the canonical projection, and $x'\in X$ is a lift of $x$, defines a network on $\Gamma\backslash X$, called the \emph{quotient network} of $\mc N$, and denoted by $\Gamma\backslash \mc N$.

    Moreover, the walk on $\Gamma\backslash \mc N$ has the same law as $(\pi(Z_k))_k$, where $(Z_k)_k$ is the walk on $\mc N$. In particular, recurrence of $\Gamma\backslash\mc N$ is equivalent to recurrence of some (equivalently, any) orbit of $\Gamma$ in $X$.
\end{prodfn}

\begin{proof}
    First of all, note that the definition of $a'$ is finite and does not depend on the choice of $x'$ by $\Gamma$-invariance.
    Also,
    \begin{eqnarray*}
        a'(y,x)&=&\frac1{|\Gamma_y|}\sum_{x'\in\pi\inv(x)}a(y',x')= \frac1{|\Gamma_y||\Gamma_x|}\sum_{\gamma\in\Gamma}a(y',\gamma x')\\
        &=&\frac1{|\Gamma_x||\Gamma_y|}\sum_{\gamma\in\Gamma} a(x',\gamma\inv y')=a'(x,y)
    \end{eqnarray*}
    and
    $$m'(\pi(x))=\frac1{|\Gamma_x|}\sum_{y\in \Gamma\backslash X}a'(\pi(x),y)=\frac1{|\Gamma_x|}\sum_{y'\in X}a(x,y')=\frac{m(x)}{|\Gamma_x|}<\infty$$
    so that the previous definition indeed defines a network.

    For the moreover part, if $Q$ is the transition matrix associated to $\Gamma\backslash\mc N$, then
    \begin{eqnarray*}
        q(\pi(x),\pi(y))&=&\frac{a'(\pi(x),\pi(y))}{m'(\pi(x))}=\sum_{y'\in\Gamma\cdot y}\frac{a(x,y')}{m(x)}\\
        &=&\sum_{y'\in\Gamma\cdot y}p(x,y')=\m P_x(\pi(Z_1)=\pi(y))
    \end{eqnarray*}
    as claimed.
\end{proof}

\begin{proof}[Proof of Proposition \ref{pro: (positive) recurrence criterion}]
    Let $Q$ be the transition matrix associated to $\Gamma\backslash\mc N$. By Theorem 5.5.12 of \cite{Dur19}, if $Q$ admits a stationary probability measure, then the Markov chain is positive recurrent, and in particular recurrent. Here $m'$ is stationary and of finite volume by assumption, so we are done.
\end{proof}

In order to prove integrability results of some random walks on groups later on, we need to prove integrability results of the return time for the quotient walk.

\begin{pro}\label{pro: exp integrable return time}
    Assume $\Gamma$ acts cocompactly on $X$. Then for any $x\in \Gamma\backslash X$, there is some $c>0$ such that
    $$\m E_x(e^{cT_x})<\infty,$$
    where $T_x$ is the first return time to $x$.
\end{pro}

\begin{proof}
    We prove that there are some $a<1$ and $N\in\m N$ such that for all $y\in\Gamma\backslash X$, $\m P_y(T_x\ge Nn)\le a^n$, which is enough.
    
    By irreducibility, for any $y\in\Gamma\backslash X$, there are some $N_y\in\m N$ and $\alpha_y>0$ such that ${p^{(N_y)}(y,x)=\alpha_y}$. In particular, $\m P_y(T_x\le N_y)\ge\alpha_y$. Since there are finitely many states, take $N$ maximal among the $N_y$, and $\alpha$ minimal among the $\alpha_y$, so that now for any $y$, 
    $$\m P_y(T_x\le N)\ge \m P_y(T_x\le N_y)\ge \alpha_y\ge \alpha.$$ 
    But now,
    \begin{eqnarray*}
        \m P_y(T_x\ge(n+1)N\:|\:T_x\ge nN)&=&\sum_{z\in\Gamma\backslash X} p^{(nN)}(y,z)\m P_z(T_x\ge N)\\
        &\le& \sum_{z\in\Gamma\backslash X} p^{(nN)}(y,z)(1-\alpha)=1-\alpha.
    \end{eqnarray*}
    Multiplying this inequality $n$ times, we indeed get $\m P_y(T_x\ge Nn)\le a^n$, with $a=1-\alpha$.
\end{proof}

\begin{rem}\label{rem: integrable return time}
Of course, if $\Gamma$ is acting transitively instead of cocompactly, $T_x=1$ a.s., so that the random walk on $\Gamma$ constructed in Section \ref{ssec: Induced walk} is actually of bounded range.

In the non-uniform case, the integrability result should be related to the speed at which $1/|\Gamma_x|$ decreases, although it is unclear what the relation should exactly be. However, since the proof of Proposition \ref{pro: (positive) recurrence criterion} shows that the quotient walk is actually positive recurrent, we have $\m E_x(T_x)=\frac {|\Gamma_x|}{m(x)}<\infty$, so the return time is at least integrable.
\end{rem}

\subsection{The Poisson-Furstenberg boundary of a lattice}
In this section we prove the following result, of independent interest:

\begin{theo}\label{theo: B_PF lattice}
    Let $X$ be a graph, and $P$ be an irreducible reversible Markov chain with bounded total conductance, and let $\Gamma<\Aut(X,P)$ be a lattice. Then there exists an admissible measure $\mu$ on $\Gamma$ such that 
    $$B_{PF}(X,P)\simeq B_{PF}(\Gamma,\mu),$$
    where the isomorphism is $\Gamma$-equivariant, and $\mu$ is symmetric whenever $P$ is.
\end{theo}

\begin{proof}
    Fix once and for all a root vertex $o\in X.$

    By Proposition \ref{pro: (positive) recurrence criterion} and Theorem \ref{theo: inducedRW}, we may induce the random walk on $X$ to a random walk $Q$ on $\Gamma\cdot o$, verifying 
    $$H^\infty(X,P)\simeq H^\infty(\Gamma\cdot o,Q).$$
    
    Moreover, this random walk is symmetric whenever $P$ is, and $\Gamma$-invariant in virtue of the ``moreover" part of Proposition \ref{theo: inducedRW}. We can now pull back this random walk to $\Gamma$, in the same way as \cite[Proposition 8.13]{Woe00}:\medskip
    
    Let $\Gamma_o$ be the stabilizer of $o$ in $\Gamma$. By definition of a lattice, it is finite. 
    
    We can then define a random walk on $\Gamma$ by letting 
    $$r(\gamma,\delta)=\frac1{|\Gamma_o|}q(\gamma o,\delta o).$$
    It is easily seen to be a random walk on $\Gamma,$ and as
    $Q$ is $\Gamma$-invariant, so is $R$, and $R$ does come from a measure $\mu$ on $\Gamma$, simply by letting $\mu(\gamma)=r(e,\gamma)$. The measure $\mu$ is symmetric whenever $Q$ and $R$ are, which is when $P$ is. Finally, as $Q$ is transitive, $R$ also is and $\mu$ is admissible.\medskip

    It is left to prove a correspondence between bounded harmonic functions for $Q$ and for $R$.
    
    $\bullet$ Let $h$ be a bounded harmonic function on $\Gamma\cdot o,$ and define $\tilde h$ on $\Gamma$ by 
    $$\tilde h(\gamma)=h(\gamma o).$$
    Then 
    \begin{eqnarray*}
    \sum_{\delta\in\Gamma}r(\gamma,\delta)\tilde h(\delta)&=& \sum_{\delta\Gamma_o\in\Gamma_o\backslash\Gamma}\sum_{\alpha\in \Gamma_o}r(\gamma,\delta\alpha)h(\delta\alpha o)\\
    &=&\sum_{\delta\Gamma_o\in\Gamma_o\backslash\Gamma}\sum_{\alpha\in \Gamma_o}\frac1{|\Gamma_o|}q(\gamma o,\delta o)h(\delta o)\\
    &=&\sum_{\delta\Gamma_o\in\Gamma_o\backslash\Gamma}q(\gamma o,\delta o)h(\delta o)\\
    &=&\sum_{y\in\Gamma\cdot o}q(\gamma o,y)h(y)=h(\gamma o)=\tilde h(\gamma),
    \end{eqnarray*}
    so that $\tilde h$ is indeed a harmonic function on $\Gamma$.\\
    $\bullet$ Now let $h$ be a harmonic function on $\Gamma$, we first prove that it is constant on the left cosets of $\Gamma_o$. Indeed, since 
    \begin{eqnarray*}
    h(\gamma)=\sum_{\delta\in\Gamma}\frac1{|\Gamma_o|}q(\gamma o,\delta o)h(\delta),
    \end{eqnarray*}
    we have, for any $\alpha\in\Gamma_o,$
    \begin{eqnarray*}
    h(\gamma\alpha)=\sum_{\delta\in\Gamma}\frac1{|\Gamma_o|}q(\gamma \alpha o,\delta o)h(\delta)=\sum_{\delta\in\Gamma}\frac1{|\Gamma_o|}q(\gamma o,\delta o)h(\delta)=h(\gamma).
    \end{eqnarray*}
    Now for some point $y=\gamma o\in\Gamma\cdot o,$ we may define $\tilde h(\gamma o)$ to be $h(\gamma)$, which does not depend on the choice of $\delta$ for which $y=\delta o,$ as those elements are precisely the left coset $\gamma\Gamma_o$ in $\Gamma$. Now
    \begin{eqnarray*}
    \sum_{y\in\Gamma\cdot o}q(\gamma o,y)\tilde h(y)&=&\sum_{\delta\Gamma_o\in\Gamma_o\backslash\Gamma}q(\gamma o,\delta o)\tilde h(\delta o)\\
    &=&\sum_{\delta\Gamma_o\in\Gamma_o\backslash\Gamma} |\Gamma_o|r(\gamma,\delta)h(\delta)\\
    &=&\sum_{\delta\Gamma_o\in\Gamma_o\backslash\Gamma} \sum_{\alpha\in\Gamma_o} r(\gamma,\delta)h(\delta)\\
    &=&\sum_{\delta\Gamma_o\in\Gamma_o\backslash\Gamma} \sum_{\alpha\in\Gamma_o} r(\gamma,\delta\alpha)h(\delta\alpha)\\
    &=&\sum_{\delta\in\Gamma}r(\gamma,\delta)h(\delta)=h(\gamma)=\tilde h(\gamma o),
    \end{eqnarray*}
    which proves harmonicity.\medskip

    The two functions we defined are clearly inverses of each other, so that we have $$H^\infty(X,P)\simeq H^\infty(\Gamma\cdot o,Q)\simeq H^\infty(\Gamma,\mu),$$
    and finally 
    $$B_{PF}(X,P)\simeq B_{PF}(\Gamma\cdot o,Q)\simeq B_{PF}(\Gamma,\mu)$$
    as desired.
\end{proof}

We now give integrability results of the measure $\mu$ we just constructed. The moments considered in this paper will be with respect to the pseudo-distance induced by the action on the graph. More precisely,
\begin{dfn}
    Let $\mu$ be a measure on $\Gamma$, and fix $o\in X$. We say $\mu$ has \emph{finite first moment} if 
    $$\int_\Gamma d(o,\gamma o)\d\mu(\gamma)<\infty.$$
    Likewise, we say $\mu$ has a \emph{finite exponential moment} if there is a constant $c>0$ such that 
    $$\int_\Gamma \exp(cd(o,\gamma o))\d\mu(\gamma)<\infty.$$
    Both of these notions (and the constant $c$) can easily be seen to be independent of the choice of the basepoint $o$.
\end{dfn}

\begin{pro}\label{pro: integrability results of walk on lattice}
    Keep the notations of Theorem \ref{theo: B_PF lattice}, and assume further that $X$ is equipped with a locally finite graph structure for which $P$ is of bounded range. Then the measure $\mu$ always has finite first moment, has a finite exponential moment when $\Gamma$ is acting cocompactly and has finite support when $\Gamma$ is acting vertex-transitively.
\end{pro}

\begin{proof}
    Let $o\in X$ be the base point giving rise to $\mu$ as in the proof of Theorem \ref{theo: B_PF lattice}. Since $P$ is of bounded range, let $A$ be such that if $P(o,x)>0$, then $d(o,x)\le A$. Let $(Z_n)_n$ be the random walk associated to $P$, and $(S_n)_n$ be the walk associated to $(\Gamma\cdot o,Q)$ given by Proposition \ref{theo: inducedRW}. Let $T$ be the return time to $\pi(o)$ for the quotient walk on $\Gamma\backslash X$. Then $d(o,Z_n)\le nA$ and $S_1=Z_T$ happen $\m P_o$-a.s., so that $d(o,S_1)\le AT$ $\m P_o$-a.s.

    Now note that if $f:X\to \m R_{\ge 0}$ is any map, then
    $$\int_\Gamma f(\gamma o)\d\mu(\gamma)= \frac1{|\Gamma_o|}\sum_{\gamma\in\Gamma}f(\gamma o)q(o,\gamma o) =\sum_{x\in\Gamma\cdot o}f(x)q(o,x)=\m E_o(f(S_1)).$$
    By Remark \ref{rem: integrable return time}, we then have that 
    $$\int_\Gamma d(o,\gamma o)\d\mu(\gamma)=\m E_o(d(o,S_1))\le A\m E_o(T)<\infty,$$
    so that $\mu$ indeed has finite first moment.

    If now $\Gamma$ is cocompact, Proposition \ref{pro: exp integrable return time} gives the existence of some $c>0$ such that 
    $$\int_\Gamma \exp(cd(o,\gamma o))\d\mu(\gamma)=\m E_o(\exp(cd(o,S_1)))\le \m E_o(e^{cAT})<\infty,$$
    so that $\mu$ indeed has a finite exponential moment.

    Finally, if $\Gamma$ is vertex-transitive, then $T=1$ a.s. and $d(o,S_1)$ is bounded a.s. But then $q(o,\cdot)$ has finite support and so does $\mu$.
\end{proof}

\subsection{Random walks on buildings}
We now apply what we have previously done, and especially Theorem \ref{theo: B_PF lattice}, to the context of buildings. For this, we need to describe the Poisson Furstenberg boundary of a building, and in fact show that under some regularity condition of the walk, the Poisson-Furstenberg boundary is the set of chambers at infinity with the harmonic measures.\medskip

There is a certain class of very nicely behaved random walks on $V_P$ called \emph{isotropic} random walks. These are random walks $P$ such that 
$$\sigma(x,y)=\sigma(x',y')\implies p(x,y)=p(x',y').$$
In other words, the probability of going from $x$ to $y$ only depends on the $P$-distance between $x$ and $y$.

The following result is known, or at least strongly believed to be true among experts, but we could not find a written proof of it, so we include one here. This proof uses the (much stronger) description of the \emph{Martin boundary} of a building.  

\begin{theo}\label{theo: Poisson boundary of building}
    Let $X$ be a regular thick affine building. Let $P$ be an isotropic random walk on $X$ with bounded range. Then the Poisson-Furstenberg boundary of $(X,P)$ is $\Aut(X)$-equivariantly isomorphic to $(\Omega,(\nu_x)_x)$, where the $\nu_x$ are the harmonic measures defined in \ref{ssec:harm meas infty}.
\end{theo}

\begin{proof}
    Since the building is thick, the spectral radius $\rho$ is strictly smaller than $1$, hence the case $\zeta=1$ we are interested in falls under the ``above the spectrum case" of Theorem A of \cite{RT21}, for which the Martin compactification $\hat{\mc M}(X)$ of $X$ is ($\Aut(X)$-equivariantly) isomorphic to the join of the visual and the Furstenberg compactifications of the building.\medskip
    
    Fix $o\in V_P$, by Theorem 10.18 of \cite{KKS66}, the walk $(Z_n)_n$ converges $\m P_o$-almost surely to some point $Z_\infty$ in $\hat{\mc M}(X)$; we let $\mu_o$ be the distribution of $Z_\infty$. By Lemma 10.42 of \cite{KKS66}, the Poisson-Furstenberg boundary of $X$ can be identified with $(\hat{\mc M}(X),\mu_o)$.
    
    By Theorem 3.9 of \cite{3.pdf}, $\sigma(o,Z_k)$ converges to a point $\xi\in\partial_\infty\Sigma$, and by \cite[Theorem 3.11]{3.pdf}, $\xi$ is in the interior of the fundamental chamber at infinity. Hence, $\mu_o$ is supported on the set 
    $$\Omega_\xi=\{x\in\mc M(X),\exists (x_n)_n\in X^\m N, x_n\tendv x\text{ and }\sigma(o,x_n)\tendv \xi\}$$
    of points in the direction $\xi$. This set is clearly in $\Aut(X)$-equivariant bijection with $\Omega$ as there is a unique point of direction $\xi$ in any $\omega\in\Omega$.
    
    Hence we may push $\mu_o$ to a measure on $\Omega$, which we will continue calling $\mu_o$. It remains only to show that the hitting distribution of $Z_\infty$ in $\Omega$ is actually $\nu_o$ to be done.
    
    To that end, let $\lambda\in P^+$, consider the partition of $\Omega$ into clopen sets given by
    $$\Omega=\bigsqcup_{x\in V_\lambda(o)}\Omega_o(x).$$
    Since $Z_n$ converges a.s. to some point in $\Omega_\xi$, there are $x\in V_\lambda(o)$, and $N\in\m N$ such that 
    $$Z_n\in A_x:=\bigcup_{\omega\in\Omega_o(x)}\mc S(x,\omega)$$
    for all $n\ge N$. Note that since the $\Omega_o(x)$ are disjoint, the $x$ verifying the above condition is unique.
    Hence if we let $k$ be the smallest such $N$, 
    $$Z_k\in A_x\iff Z_\infty \in\Omega_o(x).$$
     Moreover, $\m P_o(Z_k\in A_x)$ does not depend on the choice of $x\in V_\lambda(o)$. Indeed, since $A_x$ is the set of points $a\in V_P$ such that there is a minimal path from $o$ to $a$ going through $x$. Since $X$ is regular, we can thus find a bijection $\phi_{x,y}:A_x\to A_y$ which preserves the $\Sigma$-distance to $o$. As such, for any $x,y\in V_\lambda(o)$,
     \begin{eqnarray*}
         \m P_o(Z_k\in A_x)&=&\sum_{n\in\m N}\sum_{a\in A_x}\m P_o(Z_n=a|k=n)\m P_o(k=n)\\
         &=&\sum_{n\in\m N}\sum_{a\in A_x}\m P_o(Z_n=\phi_{x,y}(a)|k=n)\m P_o(k=n)\\
         &=&\sum_{n\in\m N}\sum_{b\in A_y}\m P_o(Z_n=b|k=n)\m P_o(k=n)=\m P_o(Z_k\in A_y)
     \end{eqnarray*}
    Hence we have 
    $$\m P_o(Z_\infty\in\Omega_o(x))=\m P_o(Z_k\in A_x)$$
    where the right-hand side does not depend on the specific $x\in V_\lambda(o)$ we chose. Finally, we get 
    $$\mu_o(\Omega_o(x))=\m P_o(Z_\infty\in\Omega_o(x))=\frac 1{N_\lambda},$$
    which proves $\mu_o=\nu_o$ since the above equality holds for all $\lambda\in P^+$.
\end{proof}

Combining the previous theorem with Theorem \ref{theo: B_PF lattice} we get the following:

\begin{theo}\label{theo: discretization of walk}
    Let $X$ be a regular thick affine building, and let $\Gamma$ be a lattice in $X$, in the sense of Definition \ref{dfn: réseau dans un graphe}. Then there is a symmetric probability measure $\mu$ on $\Gamma$ such that 
    $$B_{PF}(\Gamma,\mu)\simeq (\Omega,(\nu_x)_x).$$
    Moreover, the isomorphism is $\Gamma$-equivariant, and in particular $(\Omega,\nu_x)$ is a boundary of $\Gamma$.
\end{theo}

\section{Strong boundaries and Weyl groups}\label{sec: Strong boundaries and Weyl groups}
Often times, rigidity properties come from a careful study of a \emph{boundary} object associated to the group or the action considered. Here our boundary will be the Poisson-Furstenberg boundary of our building (or building lattice), although the reasoning here can be done in the slightly more general setting of a \emph{strong boundary}, as defined in the works of Bader and Furman. After a couple definitions, and examples, the theory of prouniform measures developed in \cite{BFL23} is used to define a new family of measures on the boundary, and these measures are then used to prove Proposition \ref{pro: the special subgroups are the special subgroups}, essentially proving that the special subgroups of the Weyl group in the sense of Bader-Furman are the special subgroups of the Weyl group in the classical sense.

\subsection{Definitions and examples}
In order to define the notion of a strong boundary, we need the following definition, see for instance \cite{GW16}.
\begin{dfn}\label{dfn: ergo isom}
    A measure-class preserving action of a second countable locally compact group $G\curvearrowright (X,\mu)$ is said to be \emph{ergodic with isometric coefficients}, or \emph{isometrically ergodic}, if for any action of $G$ on some separable metric space $(Z,d)$ on which $G$ acts by isometries, any $G$-equivariant measurable map 
    $$\phi:X\to Z$$
    is essentially constant.
\end{dfn}

\begin{rem}
    Note that isometric ergodicity implies ergodicity: to see this, take $Z=\{0,1\}$ with the trivial $G$-action.
    
    One might also think that isometric ergodicity is a very strong condition, but it is in fact implied by double ergodicity, i.e. ergodicity of the diagonal action on $X\times X$. Indeed, let $(Z,d),\phi$ be as in Definition \ref{dfn: ergo isom}. Then the map $\Phi:X\times X\to \m R$ defined by $\Phi(x,y)=d(\phi(x),\phi(y))$ is $G$-invariant measurable, hence essentially constant by double ergodicity. If the essential value of $\Phi$ is $0$, $\phi$ is indeed essentially constant. Otherwise, the essential image of $\phi$ is discrete, hence countable by separability. But then $\phi_*\mu$ is atomic, so that $\phi(x)=\phi(y)$ happens with positive $\mu\times \mu$-measure, and $\Phi$ takes the value $0$ on a positive measure set, a contradiction.
\end{rem}

Recall also the notion of an amenable action in the sense of Zimmer (see \cite{Zim84}). This property is most often used in the following way: if $G$ is a lcsc group acting by measure-class preserving automorphisms on a standard Borel space $(X,\mu)$ and if the action is amenable, then for any compact metric $G$-space $K$, there is a $G$-equivariant map $\phi:X\to\mathrm{Prob}(K)$.

Following \cite{GHL22}, we will also say that an action of $G$ on a discrete space $\m D$ is \emph{universally amenable} if for any measure $\mu$ for which the action of $G$ is measure-class preserving, the action on $(\m D,\mu)$ is (Zimmer-)amenable.

\begin{dfn}\label{dfn: strong boundary}
    Let $G$ be a locally compact group, and let $(B,\mu)$ be a standard probability space equipped with a measure-class preserving $G$-action. The space $(B,\mu)$ is a \emph{strong boundary} for $G$ if
    \begin{enumerate}
        \item The action $G\curvearrowright B$ is amenable;
        \item The diagonal action $G\curvearrowright B\times B$ is isometrically ergodic.
    \end{enumerate}
\end{dfn}

\begin{ex}\label{ex: BPF is strong boundary}
    As shown in \cite[Proposition 2.7]{BF14}, a general source of examples of strong boundaries are Poisson-Furstenberg boundaries. Indeed, let $G$ be a lcsc group, and $\mu$ be a spread-out symmetric measure on $G$, and let $(B,\nu)$ be the Poisson-Furstenberg boundary of $(G,\mu)$. Then $(B,\nu)$ is a strong boundary of $G$.
\end{ex}

In fact this is the only example of strong boundary we will need, since the previous part of this work was on describing certain Poisson-Furstenberg boundaries.\medskip

The following generalization of the Weyl group was first introduced by Bader and Furman, see for instance \cite{BF14}.
\begin{dfn}
    Let $B$ be a strong boundary of $G$. The \emph{generalized Weyl group} $W_{G,B}$ is the group of $G$-equivariant measure-class preserving automorphisms $\Aut_G(B\times B)$ of $B\times B$.
    
    We will say that a subgroup $W'< W_{G,B}$ is \emph{special} if there exists a standard Borel space $C$ equipped with a Borel $G$-action, and a $G$-equivariant Borel map $\pi : B \to C$ such that 
    $$W' = \{w \in W_{G,B}|\pi\circ \pr_1 \circ w = \pi\circ \pr_1\}.$$
    We write $W' = W_{G,B}(\pi)$ in this situation.
    
    If $H<W_{G,B}(\pi)$ is a subgroup which contains $w_{\rm{flip}}:(b,b')\mapsto (b',b)$, we will also say that $H'<H$ is special in $H$ if $H'=H\cap W_{G,B}(\pi)$ for some $\pi$.
\end{dfn}

The following was proven in \cite{BF14} in the case of real Lie groups, and in the general case for a slightly different notion of a boundary in \cite{BFS06}. A complete proof is given in \cite[Proposition 4.7]{GHL22}.

\begin{ex}
    If $G=G(k)$ is the rational points of a semisimple algebraic group over a local field $k$, the usual Weyl group is the generalized Weyl group for $B=G/P$ where $P$ is a minimal parabolic subgroup. Hence the notion indeed generalizes the Weyl group to a possibly non-linear setting.
\end{ex}

In the context of Weyl groups, Theorem \ref{theo: discretization of walk} can be restated as follows:

\begin{theo}\label{theo: strong bnd & Weyl gp of bdg lattice}
    If $X$ is a regular thick irreducible Euclidean building and $\Gamma$ is a lattice thereof, $(\Omega,\nu)$ is a strong boundary of $\Gamma$, and $W_\Sigma<W_{\Gamma,\Omega}$ is a subgroup of the generalized Weyl group.
\end{theo}

\begin{proof}
    By Theorem \ref{theo: discretization of walk}, the space $(\Omega,\nu_x)$ of chambers at infinity with the harmonic measures is a Poisson boundary of $\Gamma$, and by Example \ref{ex: BPF is strong boundary}, this implies that $(\Omega,\nu)$ is a strong boundary of $\Gamma$. Moreover, we proved in Corollary \ref{Weyl preserve la mesure} that the spherical Weyl group $W_\Sigma$ of $\Sigma$ acts on $\Omega\times\Omega$ in a measure-class preserving way, and the actions of $W_\Sigma$ and $\Aut(X)$ clearly commute, so that $W_\Sigma$ is indeed a subgroup of the generalized Weyl group for the strong boundary $(\Omega,\nu)$ of $\Gamma$.
\end{proof}

\subsection{Prouniform measures on the boundary of a building}

We now set out to determine which subgroups of $W_\Sigma$ are special. In the process, we find ourselves working with sets of measure $0$ (the subsets of $\Omega^2$ with $\delta(\omega,\omega')=w$), which is a problem, given that we are working with maps define a.e. In order to circumvent this, we introduce new prouniform measures on $\Omega$ and $\Omega^2$.

\subsubsection{The everywhere case}
We start by motivating the introduction of these objects by providing a proof in the case where ``almost everywhere" is replaced with ``everywhere".

\begin{pro}\label{pro: special subgroups, everywhere case}
    Let $\pi:\Omega\to\{0,1\}$ be any map and let $\delta:\Omega\times\Omega\to W$ be the Weyl distance. Let $E\subset W$ be the set of all $w\in W$ such that for all $\omega,\omega'\in \Omega$ with $\delta(\omega,\omega')=w$, we have $\pi(\omega)=\pi(\omega')$.

    If $J\subset\llbracket1,n\rrbracket$ is minimal such that $E\subset W_J$, then $\pi$ is constant on $J$-residues.
\end{pro}

We start by proving the following lemma.

\begin{lem}\label{lem: splitting in everywhere case}
    Let $w\in E$, and let $i,w'$ be such that $w=s_iw'$ and $\ell(w)=\ell(w')+1$. Then $s_i\in E$ and $w'\in E$.
\end{lem}

\begin{proof}
    Let $\omega\in\Omega$, and let $\mc R$ be the residue of type $i$ containing it. Let $\omega_1,\omega_2\in\mc R$ be different from $\omega$. Let $\omega'$ be such that $\delta(\omega,\omega')=w'$, then $\delta(\omega_1,\omega')=\delta(\omega_2,\omega')=w$, so that 
    $$\pi(\omega_1)=\pi(\omega')=\pi(\omega_2).$$
    Swapping the roles of $\omega_1$ and $\omega$, we obtain through the same reasoning (but a different $\omega'$) that $\pi(\omega)=\pi(\omega_2)$ as well.
    Our reasoning holds for any $\omega_{1,2}\in\mc R$, so that $\pi$ is constant on $i$-residues, which proves that $s_i\in E$.

    We have just proven that $\pi(\omega)=\pi(\omega_1)$, and since $\pi(\omega')=\pi(\omega_1)$, we also have $\pi(\omega)=\pi(\omega')$. This of course holds for every $\omega'$ such that $\delta(\omega,\omega')=w'$, and $w'\in E$ as well.
\end{proof}

\begin{proof}[Proof of Proposition \ref{pro: special subgroups, everywhere case}]
    Applying the previous lemma inductively proves that if $w=s_{i_1}\cdots s_{i_k}$ is a reduced expression, then $s_{i_j}\in E$ for all $j\le k$. If for some $i\in J$ there were no $w\in E$ for which a reduced expression contained $s_i$, we would have $E\subset W_{J\setminus \{i\}}$. Hence $s_i\in E$ for all $i\in J$, thus $\pi$ is constant on all residues of type $J$.
\end{proof}

\subsubsection{The almost everywhere case}

In order to repeat this argument in the measurable setting, we will need a measure on the set of chambers at distance $w$ from each other, which is $\nu_o$-negligible for $w\ne w_0$. This is done with the help of prouniform measures (see Section \ref{ssec: Prouniform measures}).\medskip

Let $w$ be an element of the spherical Weyl group, $\Lambda$ be the fundamental sector of $\Sigma$, and $\Lambda^w$ be the convex hull of $\Lambda$ and $w\Lambda$ in $\Sigma$.
Note that an embedding $\phi:\Lambda^w\to X$ is entirely determined by $\phi(0)$, $\partial_\infty\phi(\Lambda)\in\Omega$ and $\partial_\infty\phi(w\Lambda)$. This allows us to construct the embedding 

\begin{eqnarray*}
    i_w:\Isom(\Lambda^w,X)^o&\hookrightarrow& \Omega\times\Omega\\
    \phi&\mapsto &(\partial_\infty\phi(\Lambda), \partial_\infty\phi(w\Lambda)).
\end{eqnarray*}

For $\omega\in\Omega,$ we may define $\mu_{\omega,o}^w$ to be the restricted prouniform measure on the subset of isometries of $\Isom(\Lambda^w,X)^o$ which send $\Lambda$ to $\mc S(o,\omega)$: it is in a sense the uniform measure on the set $\Omega_{\omega,o}^w$ of chambers at distance $w$ from $\omega$, with the added condition that $\omega$ and $\omega'$ must be contained in an apartment which also contains $o$.

Under the embedding we just defined, 
$$(i_w)_*\mu_{\omega,o}^w= \delta_\omega\times\nu_{\omega,o}^w$$
for some measure $\nu_{\omega,o}^w$ on $\Omega$. 

\begin{lem}\label{lem: nu_o est absorbante}
    For any $o\in X,$ and any $w\in W_\Sigma$, we have
    $$\int_\Omega\nu_{\omega,o}^w\d\nu_o(\omega)=\nu_o.$$
\end{lem}

\begin{proof}
    By Lemma \ref{lem: desintegration des prouniformes}, we have 
    $$\int_\Omega\mu_{\omega,o}^w\d\nu_o(\omega)=\mu_o^w$$
    where $\mu_o^w$ is the restricted prouniform measure on $\Isom(\Lambda^w,X)^o$. Now since $\nu_{\omega,o}^w=(\mathrm{pr}_2\circ i_w)_*\mu_{\omega,o}^w$ and $(\mathrm{pr}_2\circ i_w)_*$ is linear and continuous ($\Omega\times \Omega$ is compact), we need only show that 
    $$(\mathrm{pr}_2\circ i_w)_*\mu_o^w=\nu_o.$$
    But if we now consider the embedding $j:\Lambda\hookrightarrow\Lambda^w$ which sends $\Lambda$ to $w\Lambda$, we have $j^*=\mathrm{pr}_2\circ i_w$, so that the equality $(\mathrm{pr}_2\circ i_w)_*\mu_o^w=\nu_o$ now follows from Corollary \ref{cor: i star star}.
\end{proof}

\begin{lem}\label{lem: desintegration des nu_o}
    Let $w,w'\in W$ be such that $\ell(ww')=\ell(w)+\ell(w')$. Then for any $o\in X$ and any $\omega\in\Omega$,
    $$\nu_{\omega,o}^{ww'}=\int_{\Omega}\nu_{\omega',o}^{w'}\d \nu_{\omega,o}^w(\omega')$$
\end{lem}

\begin{proof}
    This formula looks a lot like Lemma \ref{lem: desintegration des prouniformes}, with a couple changes. In this context, since $\Lambda^w\subset \Lambda^{ww'}$, Lemma \ref{lem: desintegration des prouniformes} reads as follows:
    $$\mu_{\omega,o}^{ww'}=\int_{\Isom(\Lambda^w,X)^o}\mu^\phi\d\mu_{\omega,o}^w(\phi),$$
    where $\mu^\phi$ is the prouniform measure on $\Isom(\Lambda^{ww'},X)^\phi$ restricted to the embeddings $\phi'$ coinciding with $\phi$ on $\Lambda^w$. 
    
    Then such an embedding $\phi'$ is determined by its restriction to $\Conv(w\Lambda,ww'\Lambda)=w\Lambda^{w'}$, and by the assumption $\ell(ww')=\ell(w)+\ell(w')$, any embedding of $w\Lambda^{w'}$ into $X$ coinciding with $\phi$ on $w\Lambda$ extends $\phi$ to some $\phi'\in\Isom(\Lambda^{ww'},X)$.
    
    Hence we have a measure-preserving isomorphism $\Phi_1$ between the set $\Isom(\Lambda^{ww'},X)^{\phi}$ of embeddings coinciding with $\phi$ on $\Lambda^w$, and the set of embeddings $\Isom(w\Lambda^{w'},X)^{\phi|_{w\Lambda}}$ coinciding with $\phi$ on $w\Lambda$. 
    
    Now note that the precomposition by $w$ also induces an isomorphism $w^*$ between $\Isom(w\Lambda^{w'},X)^{\phi|_{w\Lambda}}$ and $\Isom(\Lambda^{w'},X)^{(\phi\circ w)|_\Lambda}$. We let $\Phi:\Isom(\Lambda^{ww'},X)^\phi\to \Isom(\Lambda^{w'},X)^{(\phi\circ w)|_\Lambda}$ be defined by $\Phi:=w^*\circ\Phi_1$.
    
    Note also that if $i_w(\phi)=(\omega,\omega')$, we have $\Phi_*\mu^\phi=\mu_{\omega',o}^{w'}$, and if additionally $i_{ww'}(\phi')=(\omega,\omega'')$, then $i_{w'}(\Phi(\phi'))=(\omega',\omega'')$, so that 
    $$\pr_2\circ i_{ww'}=\pr_2\circ i_{w'}\circ\Phi.$$
    From this we obtain
    \begin{eqnarray*}
        \nu_{\omega,o}^{ww'}&=&(\pr_2\circ i_{ww'})_*\mu_{\omega,o}^{ww'}\\
        &=&(\pr_2\circ i_{w'})_*\int_{\Isom(\Lambda^w,X)^o}\Phi_*\mu^\phi\d\mu_{\omega,o}^w(\phi)\\
        &=&(\pr_2\circ i_{w'})_*\int_{\Isom(\Lambda^w,X)^o} \mu_{\pr_2(i_w(\phi)),o}^{w'}\d\mu_{\omega,o}^w(\phi)\\
        &\overset{(*)}{=}&(\pr_2\circ i_{w'})_*\int_{\Omega}\mu_{\omega',o}^{w'}\d \nu_{\omega,o}^w(\omega')\\
        &=&\int_{\Omega}\nu_{\omega',o}^{w'}\d \nu_{\omega,o}^w(\omega'),
    \end{eqnarray*}
    where $(*)$ is a change of variable $\omega'=\pr_2\circ i_w(\phi)$.
\end{proof}

If $u$ is the facet of cotype $i$ of $\omega\in\Omega$, we may view $\Res(u)$ as the set of $\omega'\in\Omega$ which are at distance $s_i$ from $\omega$. Doing this we are forgetting $\omega$ itself, which is in $\Res(u)$ but is not at distance $s_i$ from itself. However this exception is a single point, which is $(\proj_u)_*\nu_o$-negligible. In fact, we may generalize the projections $\proj_u$ to any $w\in W$, allowing us to have another description of $\nu_{\omega,o}^w$.\medskip

We define the map $\proj_\omega^w$ ($\nu_o$-almost everywhere) in the following way: for any $\omega'\in\Omega_{\op}(\omega)$, there is a unique chamber $\omega''$ in the apartment defined by $(\omega,\omega')$ such that 
$$\delta(\omega,\omega'')=w,$$
and we let $\proj_\omega^w(\omega'):=\omega''.$

Note that when $w=s_i$ and $u$ is the facet of cotype $i$ of $\omega$, the maps $\proj_u$ and $\proj_\omega^{s_i}$ do indeed coincide ($\nu_o$-almost everywhere).

\begin{lem}\label{lem: nuw est pushforward de nu}
    For any $o\in V_Q$, $w\in W$ and $\omega\in\Omega$, we have 
    $$\nu_{\omega,o}^w=(\proj_{\omega}^w)_*(\nu_o|_{\Omega_o^{\op}(\omega)})$$
    where $\Omega_o^{\op}(\omega)$ is the set of $\omega'\in\Omega^{\op}(\omega)$ such that the apartment containing $\omega$ and $\omega'$ also contains $o$, and $\mu|_A$ denotes the conditional measure, defined for a measurable $A$ such that $\mu(A)>0$ by 
    $$\mu|_A(B)=\frac{\mu(B\cap A)}{\mu(A)}.$$
\end{lem}

\begin{proof}
    We first prove the result in the case $w=w_0$. In that case, $\proj_\omega^{w_0}\equiv \id$ $\nu_o$-a.e., so that we are left to prove that $\nu_{\omega,o}^{w_0}=\nu_o|_{\Omega_o^{\op}(\omega)}$.
    
    For this, by definition of the prouniform measure, it is enough to prove that for $\lambda\in P^+$ fixed and $y\in V_\lambda(o)$, the quantity
    $$\nu_o(\Omega_o(y)\cap \Omega_o^\op(\omega))$$
    does not depend on $y$, only on $\lambda$, as long as 
    \begin{equation}
        \Omega_o(y)\cap \Omega_o^\op(\omega)\ne\emptyset.\label{eqn: shadow of y opposite to omega}
    \end{equation}
    But now, notice that (\ref{eqn: shadow of y opposite to omega}) happens if and only if $\proj_{\lk_X(o)}(y)$ is opposite to $\proj_{\lk_X(o)}(\omega)$, in which case $\Omega_o(y)\subset \Omega_o^\op(\omega)$ by Lemma \ref{lem: opp ds link ->o ds apt}, so that $\nu_o(\Omega_o(y)\cap \Omega_o^\op(\omega))=\nu_o(\Omega_o(y))=\frac1{N_\lambda}$, which does not depend on $y\in V_\lambda(o)$.\medskip

    Now for a general $w$, consider the identity embedding $j:\Lambda^w\hookrightarrow\Sigma$, and notice that the following diagram commutes:
    \[\begin{tikzcd}
	\Isom(\Sigma,X)^o && \Isom(\Lambda^w,X)^o \\
	\\
	\Omega\times\Omega && \Omega\times\Omega\\
    \\
    \Omega&&\Omega
	\arrow["j^*"', from=1-1, to=1-3]
	\arrow["i_{w_0}", hook', from=1-1, to=3-1]
	\arrow["i_w", hook', from=1-3, to=3-3]
	\arrow["\id\times \proj_\omega^w"', from=3-1, to=3-3]
    \arrow["\pr_2", from=3-1, to=5-1]
    \arrow["\pr_2", from=3-3, to=5-3]
    \arrow["\proj_\omega^w"', from=5-1, to=5-3]
\end{tikzcd}\]
    In particular, since Corollary \ref{cor: i star star} gives us that 
    $$(j^*)_*\mu_{\omega,o}^{w_0}=\mu_{\omega,o}^w,$$
    we get
    $$(\proj_\omega^w)_*\nu_{\omega,o}^{w_0}=\nu_{\omega,o}^w,$$
    and we are done.
\end{proof}

At some point, we will want to relate the measure $\nu_{\omega,o}^w$ with the measure $(\proj_\omega^w)_*\nu_o$ itself. Although this is unclear in general, there are special values of $w$ for which it is feasible, which will be enough for us.

We now study further the measures on residues and on the set of residues, which will be useful in establishing Proposition \ref{pro: the special subgroups are the special subgroups} by relating them to the $\nu_{\omega,o}^w$.

Let $J\subset\llbracket1,n\rrbracket$, and let $\Omega_J$ be the space of residues of type $J$ in $\Omega$, with projection $\pi_J:\Omega\to\Omega_J$. Note that $\Omega_J$ can be identified with the set of simplices of cotype $J$ of $\partial X$, as any chamber of some fixed residue of type $J$ has the same face of cotype $J$. Under this identification, $\pi_J$ becomes the map which associates to some $\omega\in\Omega$ its face of cotype $J$.

Let also $\Lambda_J$ be the simplicial face of $\Lambda$ in $\Sigma$ associated to the face of \emph{cotype} $J$ of $\partial \Lambda$. Then any residue of type $J$ can be identified with a simplex of cotype $J$ in $\partial X$, which, again, can itself be identified with an isometry $\phi:\Lambda_J\to X$ with $\phi(0)=o$. 

We start by recording the following consequence of Corollary \ref{cor: i star star}:

\begin{lem}
    Let $\nu_J$ be the restricted prouniform measure on $\Isom(\Lambda_J,X)^o$. Then 
    $$\nu_J=(\pi_J)_*\nu_o.$$
\end{lem}

\begin{proof}
    In order to apply Corollary \ref{cor: i star star}, we need only prove that if $i_J:\Lambda_J\to\Lambda$ is the inclusion, then $\pi_J=(i_J)^*$. This comes from the fact that we only consider type-preserving isometries $\phi:\Lambda\to X$, so $\phi(\Lambda_J)$ must be a face of $\partial \phi(\Lambda)$ of the same type as $\partial \Lambda_J$, i.e. of cotype $J$, as required.
\end{proof}

\begin{lem}
    Let $\mc R$ be a residue of type $J$, $u$ the associated face of cotype $J$ and $\psi\in\Isom(\Lambda_J,X)^o$ be the isometry centred at $o$ with $\partial\psi(\Lambda_J)=u$. Then the restricted prouniform measure on $\Isom(\Lambda,X)^\psi$ is $(\proj_\mc R)_*\nu_o$.
\end{lem}

\begin{proof}
    This proof will work the same as that of \cite[Lemma 5.4]{BFL23}, as the statements are analogous. Fix $\lambda\in P$ and $z\in V_\lambda(o)$. By definition of the prouniform measure, it is enough to prove that $\nu_o((\proj_\mc R)\inv(\Omega_o(z))$ only depends on $\lambda$ and not on $z$, as long as $\Omega_o(z)\cap\mc R\ne \emptyset$. 
    
    For $x\in V_P$, let $\Omega^{x,u}$ be the set of chambers $\omega\in\Omega$ for which there exists an apartment in $X$ which contains $x,u$ and $\omega$. We first start by proving that 
    $$\nu_o((\proj_\mc R)\inv(\Omega_o(z))\cap \Omega^{o,u})$$
    does not depend on $z$, only on $\lambda$, as long as $\Omega_o(z)$ and $\mc R$ intersect non-trivially.

    For this, let $\mu_\mathscr F^o$ be the restricted prouniform measure on $\Isom(\Sigma,X)^o$. By Corollary \ref{cor: i star star}, restricting an isometry in $\Isom(\Sigma,X)^o$ to $\Lambda$ pushes the measure $\mu_\mathscr F^o$ onto $\nu_o$. Note also that $\nu_o$-a.e. $\omega\in\Omega$ is opposite to $\mc R$, in the sense that it belongs to a residue which is opposite to $\mc R$. In particular, the Weyl distance between $\omega$ and $\proj_\mc R(\omega)$ is a.e. constant, say equal to some $w_1\in W$ (see \cite[Lemma 5.36 and Proposition 5.114]{AB08}). Hence $\omega\in (\proj_\mc R)\inv(\Omega_o(z))\cap \Omega^{o,u}$ exactly when there is some isometry $\phi:\Sigma\to X$ with $\phi(\Lambda)=\mc S(o,\omega)$, $z=\phi(w_1\lambda)$ and $\psi(\Lambda_J)=\phi(w_0\Lambda_J)(=\phi(w_1\Lambda_J))$, so that 
    $$\nu_o((\proj_\mc R)\inv(\Omega_o(z))\cap \Omega^{o,u})=\mu_\mathscr F^o(\{\phi:\Sigma\to X|\phi(\Lambda_J)=\psi(\Lambda_J),z\in\phi(\Lambda)\}),$$
    which only depends on $\Conv(z,\psi(\Lambda_J))$, by definition of the prouniform measure $\mu_\mathscr F^o$. But this is isometric to $\Conv(\lambda,\Lambda_J)$, as we assumed that $\Omega_o(z)$ intersected $\mc R$ nontrivially, which only depends on $\lambda$ and not on $z$, as hoped.\medskip

    Proving that this is enough to conclude is the exact same as that of \cite{BFL23}, the reader is referred to their proof for this part.
\end{proof}

Given the two previous lemmas, the following corollary is now a straightforward application of Lemma \ref{lem: desintegration des prouniformes}.
\begin{cor}\label{cor: desintegration wrt pi_J}
    For $J\subset \llbracket1,n\rrbracket$, the following holds:
    $$\int_{\Omega_J} (\proj_\mc R)_*\nu_o\:\d((\pi_J)_*\nu_o)(\mc R)=\nu_o.$$
\end{cor}

This disintegration formula for $\pi_J$ allows us to prove another useful formula:
\begin{cor}\label{cor: nu_R est absorbante}
    Let $J\subset\llbracket1,n\rrbracket$, and let $w\in W_J$. Let $\mc R$ be a residue of type $J$, and denote $\nu_\mc R=(\proj_\mc R)_*\nu_o$ and $\nu_J=(\pi_J)_*\nu_o$ for simplicity. Then for $\nu_J$-a.e. $\mc R\in\Omega_J$, 
    $$\int_\mc R \nu_{\omega,o}^w\d\nu_\mc R(\omega)=\nu_\mc R.$$
\end{cor}

\begin{proof}
    By Lemma \ref{lem: nu_o est absorbante},
    $$\int_\Omega \nu_{\omega,o}^w\d\nu_o(\omega)=\nu_o,$$
    so by the previous corollary,
    $$\int_{\Omega_J}\int_\Omega \nu_{\omega,o}^w\d\nu_\mc R(\omega)\d\nu_J(\mc R)=\int_{\Omega_J}\nu_\mc R\d\nu_J(\mc R).$$
    But now since for all $\omega\in\mc R$, $\nu_{\omega,o}^w$ is supported on $\mc R$, by uniqueness of the disintegration with respect to $\pi_J$, 
    $$\int_\Omega \nu_{\omega,o}^w\d\nu_\mc R(\omega)=\nu_\mc R$$
    for $\nu_J$-a.e. $\mc R$.
\end{proof}

\begin{rem}
    Note that this last formula is very similar to that of Lemma \ref{lem: nu_o est absorbante}, in that $\nu_\mc R$ is the equivalent in the smaller building $\mc R$ of $\nu_o$ in $\Omega$. The only difference between the two is that we need an affine building structure in order to define the measures on the spherical building at infinity. It turns out that there is in fact a construction of an affine building whose building at infinity is $\mc R$, called the \emph{façade} (see \cite{Rou23}) associated to the residue $\mc R$. The measure $\nu_o$ for this building is then easily seen to coincide with $\nu_\mc R$ in that context, so that Corollary \ref{cor: nu_R est absorbante} is an application of Lemma \ref{lem: nu_o est absorbante} to this façade. In particular, the formula is in fact true for \emph{every} $\mc R\in\Omega_J$, instead of $\nu_J$-a.e. $\mc R$ like in this statement. 
    
    The version we wrote here is enough for us however, and its proof does not require us to delve into the technicalities of façades.
\end{rem}

\subsection{Special subgroups of generalized Weyl groups}

Now that some basic properties of the $\nu_{\omega,o}^w$ have been established, we are ready to give a description of the special subgroups of $W$. We will need a couple preliminary lemmas. 

\begin{lem}\label{lem: chambre opp a 2 autres}
    For this Lemma, $\Delta$ is a thick \emph{spherical} building of type $(W,S)$. Then for any two chambers in $X$ there is another chamber which is opposite to both.
\end{lem}

\begin{proof}
    This proof will be algorithmic. Let $C,C'\in \Delta$, and let $C_0\in \Delta$ be opposite to $C$. Let $\delta:\Delta\times \Delta\to W$ be the Weyl distance function on $\Delta$, and $w_0\in W$ be the maximal element.
    
    There are two possible cases: 
    \begin{itemize}
        \item[-] either $\delta(C_0,C')=w_0$, in which case we are done;
        \item[-] or $\delta(C_0,C')=w$ is not maximal, and there exists $s\in S$ such that $\ell(sw)=\ell(w)+1$. Let $\mc R$ be the $s$-residue containing $C_0$, by the thickness assumption $|\mc R|\ge 3$. Note also that there is exactly one element $D$ of $\mc R$ for which $\delta(C,D)\ne w_0$, namely $\proj_{\mc R}(C)$. Hence there is at least another chamber $C_1\in\mc R$ which satisfies $\delta(C_1,C)=w_0$ and $\ell(\delta(C_1,C'))=\ell(\delta(C_0,C'))+1$.
    \end{itemize}
    This reasoning may now be repeated with $C_1$ instead of $C_0$, and so on. Since $\ell(\delta(C_i,C'))$ is bounded, we must fall in the first scenario at some point, and we are done.
\end{proof}

\begin{lem}\label{lem: invariance par w_0J}
    Let $J\subset\llbracket1,n\rrbracket$ and let $w_J$ be the longest element of $W_J$. Let $A\subset \Omega$ be measurable such that for $\nu_o$-a.e. $\omega\in\Omega$,
    $$\nu_{\omega,o}^{w_J}(A)\in\{0,1\}.\qquad(*)$$
    Then, for $\nu_o$-a.e. $\omega\in\Omega$,
    $$\m 1_A(\omega)=\nu_{\omega,o}^{w_J}(A)=(\proj_\mc R)_*\nu_o(A),$$
    where $\mc R$ is the $J$-residue containing $\omega$.
\end{lem}

\begin{proof}
    For simplicity, let $\nu_\mc R$ denote the measure $(\proj_\mc R)_*\nu_o$, and $\nu_J$ denote $(\pi_J)_*\nu_o$. By Corollary \ref{cor: desintegration wrt pi_J}, $(*)$ must hold $\nu_\mc R$-a.e. for $\nu_J$-a.e. $\mc R$. Fix now such an $\mc R$.
    
    Then for any $\omega'\in\mc R$, $\proj_{\omega'}^{w_J}\equiv \proj_\mc R$ $\nu_o$-a.e. In particular, by Lemma \ref{lem: nuw est pushforward de nu},
    $$\nu_{\omega',o}^{w_J}(A)= \frac1{\nu_o(\Omega_o^\op(\omega'))} \nu_o(A_\mc R\cap\Omega_o^\op(\omega')),$$
    where $A_\mc R=\proj_\mc R\inv(A)$. Also, $\omega''\in \Omega_o^\op(\omega')$ if and only if $\proj_{\lk o}(\omega')$ and $\proj_{\lk o}(\omega'')$ are opposite by Lemma \ref{lem: opp ds link ->o ds apt}. But now since $X$ is thick, so is $\lk o$, and thus by Lemma \ref{lem: chambre opp a 2 autres}, for any $\omega',\omega''\in\Omega,$ there is some $C\in \lk o$ such that $$\Omega_o^\op(\omega')\cap \Omega_o^\op(\omega'')\supset \proj_{\lk o}\inv(C),$$ 
    and in particular 
    $$\nu_o(\Omega_o^\op(\omega')\cap \Omega_o^\op(\omega''))>0.$$
    Hence, if $\omega',\omega''$ are chosen to verify
    \begin{eqnarray}
    \nu_{\omega',o}^{w_J}(A),\:\nu_{\omega'',o}^{w_J}(A)\in\{0,1\}\label{eqn: lem: "ergodicite" de W}
    \end{eqnarray}
    and if one of them, say $\omega'$, verifies $\nu_o(A_\mc R\cap \Omega_o^\op(\omega'))=\nu_o(\Omega_o^\op(\omega'))$, then the other verifies $\nu_o(A_\mc R\cap \Omega_o^\op(\omega''))>0$, so that it must also verify $\nu_{\omega'',o}^{w_J}(A)=1$. In other words, $\nu_{\omega',o}^{w_J}(A)=\nu_{\omega'',o}^{w_J}(A)$, as long as (\ref{eqn: lem: "ergodicite" de W}) holds. As (\ref{eqn: lem: "ergodicite" de W}) holds for $\nu_\mc R$-a.e. $\omega',\omega''$, we conclude that $\omega'\mapsto\nu_{\omega',o}^{w_J}(A)$ is constant $\nu_\mc R$-a.e.

    But now by Corollary \ref{cor: nu_R est absorbante}, for $\nu_J$-a.e. $\mc R$, this value is also shared by $\nu_\mc R$. Finally, we may use Corollary \ref{cor: desintegration wrt pi_J} in the other direction to conclude that the lemma does hold for $\nu_o$-a.e. $\omega\in\Omega$.
\end{proof}

\begin{lem}\label{lem: "ergodicite" de W}
    Let $o\in X$ and $J\subset \llbracket1,n\rrbracket$. Let $A\subset \Omega$ be such that for all $i\in J$ and $\nu_o$-a.e. $\omega\in \Omega$, $\nu_{\omega,o}^{s_i}(A)\in\{0,1\}$. Then for $\nu_o$-a.e. $J$-residue $\mc R$, $(\proj_\mc R)_*\nu_o(A)\in\{0,1\}$.
    
    In particular, if $J=\llbracket1,n\rrbracket$, then $\nu_o(A)\in\{0,1\}$.
\end{lem}

\begin{proof}
    By the previous lemma with $J=\{i\}$, if we let $B_i\subset\Omega$ be the set of $\omega\in\Omega$ such that $\nu_{\omega,o}^{s_i}(A)=1$, then $A=B_i$ $\nu_o$-a.e., for any $i$. In particular, since 
    $$\int_\Omega \nu_{\omega,o}^w\d \nu_o(\omega)=\nu_o$$
    for any $w\in W$, we also obtain $\nu_{\omega,o}^w(B_i)=\nu_{\omega,o}^w(A)$ for $\nu_o$-a.e. $\omega$. But now, 
    $$\nu_{\omega,o}^{ws_i}(A)=\int_\Omega \underbrace{\nu_{\omega',o}^{s_i}(A)}_{\m 1_{B_i}(\omega')}\d \nu_{\omega,o}^w(\omega')=\nu_{\omega,o}^w(B_i)=\nu_{\omega,o}^w(A)$$
    for $\nu_o$-a.e. $\omega$, all $w\in W_J$ and all $i\in J$ such that $\ell(ws_i)=\ell(w)+1$ by Lemma \ref{lem: desintegration des nu_o}.
    
    By a straightforward induction on $\ell(w)$, we thus get that for all $w\in W_J$, $\nu_{\omega,o}^w(A)\in\{0,1\}$ for $\nu_o$-a.e. $\omega\in \Omega$.\medskip

    In particular, we have $\nu_{\omega,o}^{w_J}(A)\in\{0,1\}$ for a.e. $\omega\in\Omega$, and we are done by Lemma \ref{lem: invariance par w_0J}.
\end{proof}

This is the first part of the analogue of Lemma \ref{lem: splitting in everywhere case}...

\begin{lem}\label{lem: passage au suffixe}
    If $A\subset \Omega$ and $w\in W$ are such that for $\nu_o$-a.e. $\omega$, 
    $$\nu_{\omega,o}^w(A)\in \{0,1\},$$
    then this is also true for any suffix of $w$. That is, if $w=w_1w_2$ with $\ell(w)=\ell(w_1)+\ell(w_2)$, then we also have 
    $$\nu_{\omega,o}^{w_2}(A)\in \{0,1\}.$$
\end{lem}

\begin{proof}
    By Lemma \ref{lem: desintegration des nu_o}, we have 
    $$\nu_{\omega,o}^w(A)=\int_\Omega \nu_{\omega',o}^{w_2}(A)\d\nu_{\omega,o}^{w_1}.$$
    Since for every $\omega'\in\Omega$, $0\le \nu_{\omega',o}^{w_2}(A)\le 1$, we must have $\nu_{\omega',o}^{w_2}(A)=\nu_{\omega,o}^w(A)$ for $\nu_{\omega,o}^{w_1}$-a.e. $\omega'$ for every $\omega$ such that $\nu_{\omega,o}^w(A)\in\{0,1\}$. 
    
    As this happens for $\nu_o$-a.e. $\omega$, by Lemma \ref{lem: nu_o est absorbante}, we do indeed have $\nu_{\omega',o}^{w_2}(A)\in\{0,1\}$ for $\nu_o$-a.e. $\omega'$.
\end{proof}

...and this is the second part of the analogue of Lemma \ref{lem: splitting in everywhere case}.

\begin{lem}\label{lem: passage au préfixe}
    Let $A\subset \Omega$ and $w\in W$, and let $i,w'$ be such that $w=w's_i$ and $\ell(w)=\ell(w')+\ell(s_i)$. Assume that for $\nu_o$-a.e. $\omega$, 
    $$\nu_{\omega,o}^w(A)\in \{0,1\}.$$
    Then this also holds for $w'$: for $\nu_o$-a.e. $\omega$,
    $$\nu_{\omega,o}^{w'}(A)\in \{0,1\}.$$
\end{lem}

\begin{proof}
    Applying Lemma \ref{lem: passage au suffixe}, we have 
    $$\nu_{\omega,o}^{s_i}(A)\in\{0,1\}$$
    for $\nu_o$-a.e. $\omega$. Now applying Lemma \ref{lem: invariance par w_0J}, we also have that for $\nu_o$-a.e. $\omega$,
    $$\nu_{\omega,o}^{s_i}(A)=\m 1_A(\omega).\qquad (*)$$
    Now by Lemma \ref{lem: nu_o est absorbante}, $(*)$ also holds for $\nu_o$-a.e. $\omega$ for $\nu_{\omega,o}^{w'}$-a.e. $\omega'$. Integrating $(*)$, we then get
    $$\nu_{\omega,o}^{w'}(A)=\int_\Omega\m 1_A(\omega') \d\nu_{\omega,o}^{w'}(\omega')=\int_\Omega\nu_{\omega',o}^{s_i}(A)\d\nu_{\omega,o}^{w'}(\omega')=\nu_{\omega,o}^{w's_i}(A)=\nu_{\omega,o}^w(A)\in\{0,1\}$$
    for $\nu_o$-a.e. $\omega$.
\end{proof}

\begin{pro}\label{pro: the special subgroups are the special subgroups}
    The special subgroups of $W_\Sigma<\Aut_\Gamma(\Omega\times \Omega)$ are exactly the $W_J$ for some $J\subset \llbracket1,n\rrbracket$.
\end{pro}

\begin{proof}
    Let $W'$ be a special subgroup, and let $\pi:\Omega\to C$ be the $\Gamma$-equivariant projection for which $W'=W_\Sigma\cap W_{\Gamma,\Omega}(\pi)$. Let also $J$ be minimal such that $W'<W_J$. For simplicity, denote $\nu_\mc R:=(\proj_\mc R)_*\nu_o$ and $\nu_J=(\pi_J)_*\nu_o$.  Our goal is to show that for $\nu_J$-a.e. $\mc R$, $\pi$ is $\nu_\mc R$-essentially constant. We would then be done, as we would thus obtain that any $w\in W_J$ satisfies $\pi\circ\pr_1\circ w=\pi\circ\pr_1$ $\nu_o$-a.e. by Corollary \ref{cor: desintegration wrt pi_J}, so that $W_J<W_\Sigma\cap W_{\Gamma,\Omega}(\pi)$. 

    Recall that for $W'$ to be a special subgroup, $\pi$ needs to satisfy 
    $$\pi(\pr_1(w\cdot(\omega,\omega')))=\pi(\omega)$$
    for $\nu_o$-a.e. $\omega,\omega'$ and all $w\in W'$. In other words, for almost every apartment of $\partial X$, all the chambers in an orbit of the Weyl group have the same image by $\pi$. By Fubini's Theorem, this is the same as saying that for $\nu_o$-a.e. $\omega$ and for $\nu_o$-a.e. apartment containing $\omega$, the chamber $\omega$ and the chamber at distance $w$ of $\omega$ in that apartment have the same image by $\pi$. 
    
    Hence if we let $A'\subset C$ be Borel, and let $A=\pi\inv(A')$, then for $\nu_o$-a.e. $\omega$ and all $w\in W'$,
    $$(\proj_\omega^w)_*\nu_o(A)=\m 1_A(\omega)\in \{0,1\}.$$
    Now since $\nu_o|_{\Omega_o^{\op}(\omega)}\ll \nu_o$ and by Lemma \ref{lem: nuw est pushforward de nu}, we also have $\nu_{\omega,o}^w\ll (\proj_\omega^w)_*\nu_o$. We then also have, for all $w\in W'$ and $\nu_o$-a.e. $\omega$,
    $$\nu_{\omega,o}^w(A)\in \{0,1\}.$$
    
    Let $j\in J$, and pick $w\in W'$ such that some reduced expression of $w$ contains $s_j$, say $w=s_{i_1}\cdots s_{i_k}$ with $i_l=j$. Such a $w$ always exists, as otherwise we would have $W'<W_{J\setminus \{j\}}$.
    
    We have proven above that for $\nu_o$-a.e. $\omega$,
    $$\nu_{\omega,o}^w(A)\in \{0,1\}.$$
    We can thus apply Lemma \ref{lem: passage au préfixe} inductively to obtain that for $w'=s_{i_1}\cdots s_{i_l}$, we also have 
    $$\nu_{\omega,o}^{w'}(A)\in \{0,1\}$$
    for $\nu_o$-a.e. $\omega$.
    Now applying Lemma \ref{lem: passage au suffixe}, we have 
    $$\nu_{\omega,o}^{s_j}(A)\in \{0,1\}$$
    for $\nu_o$-a.e. $\omega$.\medskip
    
    This reasoning holds for any $j\in J$, so we can now apply Lemma \ref{lem: "ergodicite" de W} to conclude that for any Borel subset $A'\subset C$, and for almost every $J$-residue $\mc R$, $\nu_\mc R(\pi\inv(A'))\in\{0,1\}$, which does indeed prove that $\pi$ is $\nu_\mc R$-essentially constant for $\nu_J$-a.e. $\mc R$.
\end{proof}

\section{Superrigidities}\label{sec: Superrigidities}
\subsection{Geometrically rigid spaces and superrigidity for morphisms}
The general idea in this section is that building lattices cannot act on spaces with hyperbolic properties. Here ``hyperbolic properties" should be understood in a slightly more general setting than just Gromov-hyperbolic, as the proof of Propositions \ref{pro: 3 cases morphism rigidity} and \ref{pro: 3 cases cocycle rigidity} can be done for more than just hyperbolic spaces. This is formalized in the following definition, which was inspired by earlier works of Bader and Furman, and given in \cite{GHL22}. Recall that the definition of a universally amenable action was given right before Definition \ref{dfn: strong boundary}.

\begin{dfn}\label{dfn: geometrically rigid}
    Let $\Lambda$ be a countable discrete group. Let $\m D$ be a countable discrete $\Lambda$-space, let $K$ be a compact metric space equipped with a $\Lambda$-action by homeomorphisms, and let $\Delta$ be a Polish space equipped with a $\Lambda$-action by Borel automorphisms. We say that the triple $(\m D,K,\Delta)$ is \emph{geometrically rigid} if the following hold:
    \begin{enumerate}
        \item The space $K$ admits a Borel $\Lambda$-invariant partition $K = K_\mathrm{bdd} \sqcup K_\infty$ such that
        \begin{itemize}
            \item there exists a Borel $\Lambda$-equivariant map $\theta_\mathrm{bdd} : K_\mathrm{bdd} \to \mc P_{<\infty}(\m D)$,
            \item there exists a Borel $\Lambda$-equivariant map $\theta_\infty : K_\infty \to \Delta$.
        \end{itemize}
        \item There exists a Borel $\Lambda$-equivariant map bar : $\Delta^{(3)} \to \mc P_{<\infty}(\m D)$.
        \item The $\Lambda$-action on $\Delta$ is universally amenable.
    \end{enumerate}
We say that $\Lambda$ is \emph{geometrically rigid with respect to} $\m D$ if there exist a compact metric $\Lambda$-space $K$ and a Polish $\Lambda$-space $\Delta$ such that the triple $(\m D,K,\Delta)$ is geometrically rigid.
\end{dfn}

We now give some examples of geometrically rigid spaces.\medskip
\begin{itemize}
    \item If $\Lambda$ is a hyperbolic group, then $\Lambda$ is geometrically rigid to itself. Indeed, let $\Delta$ be the Gromov boundary of $\Lambda$, and $K=\Lambda\cup\Delta$. If we let $K_{\rm{bdd}}=\Lambda$ and $K_\infty=\Delta$, and $\theta_{\rm{bdd}}$ and $\theta_\infty$ be the corresponding identity maps, everything works. The existence and measurability of the barycenter map, as well as the universal amenability of the $\Lambda$-action on its boundary is due to Adams (see \cite{Ada94}).
    \item If now $H$ is a finitely generated group, relatively hyperbolic with respect to a finite collection $\mc P$ of subgroups, then $H$ is geometrically rigid with respect to a hyperbolic graph, whose vertex stabilizers are either trivial or conjugated to a subgroup in $\mc P$, see \cite[Section 5.1]{GHL22}.
    \item If $\mc G$ is a finite simple graph (that is, $\mc G$ has no loops), let $A_\mc G$ be the right-angled Artin group associated to $\mc G$. If $\mc G$ does not split non-trivially as a join, then $A_\mc G$ is geometrically rigid with respect to the set of standard subcomplexes of the Salvetti complex $\tilde S_\mc G$, see \cite[Section 5.2]{GHL22}.
    \item Let $S$ be a boundaryless connected orientable surface of finite type. Its extended mapping class group $\rm{Mod}^*(S)$ is the group of all isotopy classes of homeomorphisms of $S$ which restrict to the identity on every boundary component (where isotopies are required to fix the boundary at all times). If $S$ is neither the torus or a $p$-punctured sphere with $p\le 3$, then $\rm{Mod}^*(S)$ is geometrically rigid with respect to the vertex set of the curve graph (see \cite[Section 9.1]{GHL22} for details).
    \item Let $G=H_1*H_2*\cdots H_p*F_N$ be the free product of some infinite countable groups $H_i$ and a free group of rank $N$. We let $\mc F=\{[H_1],\ldots [H_p]\}$ be the set of conjugacy classes of the $H_i$ in $G$, and let $\Out(G,\mc F)$ be the subgroup of $\Out(G)$ which preserves all the conjugacy classes in $\mc F$. Let also $\Out(G,\mc F^{(t)})$ be the subgroup of $\Out(G,\mc F)$ made of all outer automorphisms which have a representative in $\Aut(G)$ whose restriction to each of the $H_i$ is the conjugation by an element $g_i \in G$. If for any $i$, the centralizer in $H_i$ of any non-trivial element is amenable, then $\Out(G,\mc F^{(t)})$ is geometrically rigid with respect to the set of all conjugacy classes of proper free factors of $(G,\mc F)$. See \cite[Section 9.2]{GHL22}.
\end{itemize}

Geometrically rigid spaces are used in \cite[Proposition 4.10]{GHL22} to give a criterion for cocycles to verify some sort of rigidity. The following is their result applied to $X=\{\cdot\}$, so that it may be stated without any mention of cocycles for the reader only interested in morphisms (see Proposition \ref{pro: 3 cases cocycle rigidity} for the cocycle version).

\begin{pro}\label{pro: 3 cases morphism rigidity}
    Let $G$ be a locally compact group, and let $B$ be a strong boundary for $G$. Let $\Lambda$ be a countable discrete group, let $\m D$ be a countable discrete $\Lambda$-space, and assume that $\Lambda$ is geometrically rigid with respect to $\m D$. 
    
    Then there exists a Polish space $\Delta$ equipped with a universally amenable $\Lambda$-action, such that for any morphism $f:G\to \Lambda$, either 
    \begin{enumerate}
        \item the image of $f$ setwise stabilizes a finite subset of $\m D$, or
        \item the image of $f$ stabilizes a set of cardinality at most $2$ in $\Delta$, or else
        \item there exists a morphism $\phi:W_{G,B}\to\m Z/2\m Z$ with $\phi(w_{\mathrm{flip}})\ne0$ whose kernel is a special subgroup of the Weyl group $W_{G,B}$.  
    \end{enumerate}
\end{pro}

An overview of the proof is given right after the statement of Proposition \ref{pro: 3 cases cocycle rigidity} in the more general case of cocycles.\medskip

Since the $\Lambda$-action on $\Delta$ is amenable, any point-stabilizer is amenable. In particular, in the second case of the proposition, the image of $f$ is amenable. Hence if $G$ has property (T), then so will $\rm{im}(f)$, so that $\rm{im}(f)$ is in fact finite.

We are now able to prove the morphism analogue for buildings of \cite[Theorem 2]{GHL22}:

\begin{theo}\label{theo: morphism superrigidity}
    Let $\Lambda$ be a countable group acting on a countable set $\m D$. Assume that $\Lambda$ is geometrically rigid with respect to $\m D$.
    Let $X$ be a thick irreducible Euclidean building of higher rank, or a product thereof. Let $\Gamma$ be a lattice in $X$ with property (T). 
    
    Then the image of every morphism $f : \Gamma \to \Lambda$ virtually fixes an element of $\m D$.
\end{theo}

\begin{proof}
    By Theorem \ref{theo: strong bnd & Weyl gp of bdg lattice}, $W_\Sigma$ is a generalized Weyl group for the strong boundary $(\Omega,\nu)$ of $\Gamma$. Applying Proposition \ref{pro: 3 cases morphism rigidity}, there is a Polish space $\Delta$ equipped with a universally amenable $\Lambda$-action such that we are in one of three cases.
    \begin{enumerate}
        \item If $\mathrm{im}(f)$ setwise stabilizes a finite subset $A$ of $\m D$, then the kernel of the map $\mathrm{im}(f)\to \mathrm{Sym}(A)$ is of finite index in $\mathrm{im}(f)$ and fixes a point in $\m D$.
        \item If $\mathrm{im}(f)$ stabilizes a set of cardinality at most $2$ in $\Delta$, then $\mathrm{im}(f)$ is finite, since it has (T) and is amenable, so $\mathrm{im}(f)$ virtually fixes all of $\m D$.
        \item The last possibility never happens. Indeed, assume we have a morphism $\phi:W_{\Gamma,\Omega}\to \m Z/2\m Z$ for which the kernel is a special subgroup and $\phi(w_{\mathrm{flip}})\ne0$. Then $\phi$ restricted to $W_\Sigma$ is still surjective as $w_{\mathrm{flip}}\in W_\Sigma$, and the kernel $W'$ of $\phi|_{W_\Sigma}$ is a proper special subgroup of $W_\Sigma$. By Proposition \ref{pro: the special subgroups are the special subgroups}, there is some $J\subsetneq \llbracket1,n\rrbracket$ such that $W' =W_J$. But since every factor of $X$ is of higher rank, any $W_J$ is of index larger than $2$ in $W_\Sigma$, a contradiction.
     \end{enumerate}
\end{proof}

\begin{rem}\label{rem: hyp (T) pas grave}
    Note that the assumption for $\Gamma$ to have property (T) is verified for all known examples, as Oppenheim recently proved that all uniform building lattices have (T) (see \cite{Opp24}). Property (T) also holds for non-uniform exotic lattices of type $\tilde A_2$ (see \cite{LSW23}), and is conjectured to hold for all types, though no examples of non-uniform exotic lattices are known at the moment.
\end{rem}

In specific examples of geometrically rigid spaces, this result can be strengthened to the fact that $\mathrm{im}(f)$ is simply finite. We only state the results here, their proof is the exact same as their analogues in \cite{GHL22}.

\begin{cor}\label{cor: grosse liste pour morphismes}
    Let $X$ be a product of thick irreducible Euclidean buildings of higher rank, and $\Gamma$ be a lattice in $X$ with property (T). 
    Then the following hold:
    \begin{enumerate}
        \item Let $\Lambda$ be a group which is hyperbolic relative to a finite collection $\mc P$ of subgroups. 
        Then the image of every morphism $\Gamma\to\Lambda$ is either finite or included in a conjugate of a parabolic subgroup of $\Lambda$.
        \item As a direct corollary from the first point, if $\Lambda$ is a hyperbolic group, then every morphism $\Gamma\to\Lambda$ has finite image.
        \item Let $\mc G$ be a finite simple graph, and let $A_\mc G$ be the associated right-angled Artin group.
        Then every morphism $\Gamma\to A_\mc G$ is trivial.
        \item Let $S$ be a surface of finite type. Then the image of every morphism $\Gamma\to\mathrm{Mod}^*(S)$ is finite.
        \item Let $H$ be a torsion-free hyperbolic group. Then any morphism $\Gamma\to \Out(H)$ has finite image. 
    \end{enumerate}
\end{cor}

\subsection{Cocycles}
Cocycles are a general construction, generalizing that of a morphism and arising in a multitude of situations in geometric and measured group theory. We state here most results from the previous section in the context of cocycles.
\begin{dfn}
    Let $G$ be a group, along with a measure-class preserving action on a standard Borel set $(X,\mu)$, and let $\Lambda$ be a countable discrete group.
    \begin{itemize}
        \item A map $c:G\times X\to\Lambda$ is a \emph{cocycle} if it verifies 
    $$c(g_1g_2,x)=c(g_1,g_2x)c(g_2,x)$$
    for every $g_1,g_2\in G$ and almost every $x\in X$.
        \item Two cocycles $c_1,c_2$ are said to be \emph{cohomologous} if there is a Borel map $f:X\to \Lambda$ such that for any $g$ and a.e. $x$, $$c_2(g,x)=f(gx)^{-1}c_1(g,x)f(x).$$
        \item Given a standard Borel $\Lambda$-space $\Delta$ and a cocycle $c$, a Borel map $f:X\to \Delta$ is said to be $c$-\emph{equivariant} if for all $g$ and a.e. $x$,
        $$f(gx)=c(g,x)f(x).$$
    \end{itemize}
\end{dfn}

Cocycles may be seen as a generalization of morphisms. Indeed, if $\pi:G\to\Lambda$ is a morphism, then the map $c_\pi:G\times X\to\Lambda$ defined by $c_\pi(g,x)=\pi(g)$ is a cocycle. 

Cocycles appear naturally in group cohomology, where these names are taken from, but also in the context of measure equivalence.

\begin{ex}
    Two discrete groups $\Gamma,\Lambda$ are said to be \emph{measure equivalent} (ME for short) if there is an essentially free action $\Gamma\times\Lambda\curvearrowright(\Omega,\mu)$ by measure-preserving automorphisms on some (infinite) measured space, for which both the actions of $\Lambda$ and $\Gamma$ admit fundamental domains of finite measure.
    
    Let $X$ be a fundamental domain for the $\Lambda$-action. Then by freeness of the $\Lambda$-action, for any $\gamma\in\Gamma$ and $x\in X$, there is a unique element $\lambda\in\Lambda$ which takes $\gamma x$ back to $X$, i.e. $(\gamma,\lambda)\cdot x\in X$. This allows us both to define a cocycle 
    $$c:(\gamma,x)\mapsto \lambda$$ 
    and an action of $\Gamma$ on $X$ by $\gamma\cdot x=(c(\gamma,x),\gamma)\cdot x$.
\end{ex}

We will say a pair $(G,\Lambda)$ is \emph{cocycle-rigid}, where $G$ is any lcsc group and $\Lambda$ is countable discrete, if for any standard probability space $X$ endowed with an ergodic pmp $G$-action, any cocycle $G\times X\to\Lambda$ is cohomologous to a cocycle with values in a finite subgroup of $\Lambda$. 

We will see that the machinery of geometrically rigid spaces is able to provide a number of examples of cocycle rigid pairs (see Corollary \ref{cor: grosse liste pour cocycles}).

This is the cocycle equivalent of Proposition \ref{pro: 3 cases morphism rigidity} (see also \cite[Proposition 4.10]{GHL22}):
\begin{pro}\label{pro: 3 cases cocycle rigidity}
    Let $G$ be a locally compact group, and let $B$ be a strong boundary for $G$. Let $X$ be a standard probability space equipped with an ergodic measure-preserving $G$-action. Let $\Lambda$ be a countable discrete group, let $\m D$ be a countable discrete $\Lambda$-space, and assume that $\Lambda$ is geometrically rigid with respect to $\m D$. 
    
    Then there exists a Polish space $\Delta$ equipped with a universally amenable $\Lambda$-action, such that for every cocycle $c:G\times X\to \Lambda$, either 
    \begin{enumerate}
        \item the cocycle $c$ is cohomologous to a cocycle that takes its values in the setwise stabilizer of a finite subset of $\m D$, or\label{item: 3 cases cocycle rigidity 1}
        \item there exists a $c$-equivariant Borel map $X\to\Delta$ or a $c$-equivariant Borel map $X\to\Delta^{(2)}/\mathfrak S_2$, or else\label{item: 3 cases cocycle rigidity 2}
        \item there exists a morphism $\phi:W_{G,B}\to\m Z/2\m Z$ with $\phi(w_{\mathrm{flip}})\ne0$ whose kernel is a special subgroup of the Weyl group $W_{G,B}$.\label{item: 3 cases cocycle rigidity 3}
    \end{enumerate}
\end{pro}

Here is a quick overview of the proof of this proposition. For an actual proof, see \cite[Proposition 4.10]{GHL22}. For $i\in\m N$, let $c_i:G\times X\times B^i\to\Lambda$ be defined by $c_i(g,x,b_1,\ldots b_i)=c(g,x)$. This is still a cocycle, and we can now use amenability of the $G$-action on $B$ (and thus on $X\times B$), we get a $c_1$-equivariant map $X\times B\to \rm{Prob}(K)$. By ergodicity and $G$-invariance of the decomposition $K=K_{\rm{bdd}}\sqcup K_\infty$, the image of the map is either in $\rm{Prob}(K_{\rm{bdd}})$ or in $\rm{Prob}(K_\infty)$. In the first case, composing with $(\theta_{\rm{bdd}})_*$ we get a $c_1$-equivariant map $X\times B\to \rm{Prob}(\mc P_{<\infty}(\m D))$. Using the fact that $\mc P_{<\infty}(\m D)$ is countable and a bit of work, we get conclusion \ref{item: 3 cases cocycle rigidity 1}.

If instead we have a $c_1$-equivariant map $X\times B\to \rm{Prob}(K_\infty)$, we can push it to a map $f:X\times B\to \rm{Prob}(\Delta)$. By ergodicity, it is either supported on $\rm{Prob}_{\ge 3}(\Delta)$ or $\rm{Prob}_{\le 2}(\Delta)$. In the first case, we can use the barycenter map to get conclusion \ref{item: 3 cases cocycle rigidity 1} again. In the other case, if $f$ does not depend on $B$, conclusion \ref{item: 3 cases cocycle rigidity 2} holds. Otherwise, still by ergodicity, $f$ has image either in $\rm{Prob}_1(\Delta)=\Delta$ or $\rm{Prob}_2(\Delta)$. In the second case, we again get conclusion \ref{item: 3 cases cocycle rigidity 1}. Finally, we must have a $c_1$-equivariant map $f:X\times B\to \Delta$, which depends on the $B$-coordinate. We can then again construct a $c_2$-equivariant map $f^2:X\times B^2\to \mc P_2(\Delta)$ by $f^2(x,b,b')=\{f(x,b),f(x,b')\}$. This map must be essentially unique, lest the barycenter argument be repeated again. In this case, there are exactly two $c_2$-equivariant maps $X\times B^2\to \Delta^{(2)}$, and the action on these two maps of $W_{G,B}$ induces a morphism into $\m Z/2\m Z$, for which $w_{\rm{flip}}$ acts non-trivially. One then checks that the kernel of this action is special, for $C=L^0(X,\Delta)$, giving conclusion \ref{item: 3 cases cocycle rigidity 3}.\medskip

The following result generalizes the fact that any amenable discrete group with property (T) is finite, and will be used in the same way as for the morphism case. It appears in a paper by Spatzier-Zimmer \cite{SZ91}, and a proof of this result is also given in \cite[Lemma 4.16]{GHL22} for instance.

\begin{pro}\label{pro: Spatzier-Zimmer}
    Let $G$ be a lcsc group with property (T), and let $X$ be a standard Borel probability space equipped with an ergodic measure-preserving $G$-action. Let $\Lambda$ be a countable group acting on a standard Borel space $\Delta$, so that the $\Lambda$-action on $\Delta$ is universally amenable. 
    Let $c : G \times X \to \Lambda$ be a measurable cocycle, and assume that there exists a $c$-equivariant Borel map from $X$ to $\Delta$.
    Then $c$ is cohomologous to a cocycle with values in a finite subgroup of $\Lambda$.
\end{pro}

With these results in hand, we are now able to prove the analog for buildings of \cite[Theorem 2]{GHL22}. Compare also with Theorem \ref{theo: morphism superrigidity}.

\begin{theo}\label{theo: cocycle superrigidity}
    Let $\Lambda$ be a countable group acting on a countable set $\m D$. Assume that $\Lambda$ is geometrically rigid with respect to $\m D$.
    Let $\mc X$ be a thick irreducible Euclidean building of higher rank, or a product thereof. Let $\Gamma$ be a lattice in $\mc X$ with property (T). Let $X$ be a standard probability space equipped with an ergodic measure-preserving $\Gamma$-action.
    Then every cocycle $c : \Gamma \times X \to \Lambda$ is cohomologous to a cocycle that takes its values in a subgroup of $\Lambda$ that virtually fixes an element of $\m D$.
\end{theo}

(See Remark \ref{rem: hyp (T) pas grave} on the Property (T) hypothesis for $\Gamma$.)\bigskip

\begin{proof}
    By Example \ref{theo: strong bnd & Weyl gp of bdg lattice}, $W_\Sigma$ is a generalized Weyl group for the strong boundary $(\Omega,\nu)$ of $\Gamma$. Applying Proposition \ref{pro: 3 cases cocycle rigidity}, there is a Polish space $\Delta$ equipped with a universally amenable $\Lambda$-action such that we are in one of three cases.
    \begin{enumerate}
        \item If $c$ is cohomologous to a cocycle that takes its values in the setwise stabilizer $H$ of a finite subset $A$ of $\m D$, then the kernel of the map $H\to \mathrm{Sym}(A)$ is of finite index in $H$ and fixes a point in $\m D$.
        \item If there exists a $c$-equivariant Borel map, from $X$ to either $\Delta$ or $\Delta^{(2)}/\mathfrak S_2$, then by Proposition \ref{pro: Spatzier-Zimmer}, the cocycle $c$ is cohomologous to a cocycle that takes its values in a finite subgroup of $\Lambda$, so the conclusion clearly holds.
        \item The last possibility never happens, the statement here is the exact same as that of Theorem \ref{theo: morphism superrigidity}, and so the same proof still applies.
     \end{enumerate}
\end{proof}

The results from the previous section with morphisms instead of cocycles still hold with cocycles. We sum everything up in the following corollary. The proofs all work in the exact same way as in \cite{GHL22}. 

Recall that a \emph{cocycle-rigid pair} is a pair $(G,\Lambda)$, where $G$ is any lcsc group and $\Lambda$ is a countable discrete group, if for any standard probability space $X$ endowed with an ergodic pmp $G$-action, any cocycle $G\times X\to\Lambda$ is cohomologous to a cocycle with values in a finite subgroup of $\Lambda$.

\begin{cor}\label{cor: grosse liste pour cocycles}
    Let $\Gamma$ be a lattice in a product of thick irreducible Euclidean buildings of higher rank with property (T), and let $X$ be a standard probability space equipped with an ergodic measure-preserving action of $\Gamma$.
    The following hold:
    \begin{enumerate}
        \item If $\Lambda$ is a group which is hyperbolic relative to a finite collection $\mc P$ of subgroups. Then every cocycle $\Gamma \times X \to \Lambda$ is cohomologous to a cocycle that takes its values in a parabolic or in a finite subgroup of $\Lambda$.
        \item In particular, if $\Lambda$ is hyperbolic, $(\Gamma,\Lambda)$ is cocycle rigid.
        \item Let $\mc G$ be a finite simple graph, and let $A_\mc G$ be the associated right-angled Artin group. Let $X$ be a standard probability space equipped with an ergodic measure-preserving action of $\Gamma$.
        Then every cocycle $\Gamma\times X \to A_\mc G$ is cohomologous to the trivial cocycle.
        \item Let $S$ be a surface of finite type. Then $(\Gamma,\rm{Mod}^*(S))$ is cocycle rigid.
        \item Let $H_1,\ldots ,H_p$ be countable groups, let $F_N$ be a free group of rank $N$, and let
        $G := H_1*\cdots*H_p*F_N$.
        Let $\mc F = \{[H_1],\ldots ,[H_p]\}$.
        
        Assume that
        \begin{itemize}
            \item for each $i \in \{1,\ldots,k\}$, the pair $(\Gamma, H_i/Z(H_i))$ is cocycle-rigid, and
            \item for each $i \in \{1,\ldots,k\}$, the centralizer in $H_i$ of every nontrivial element is amenable.
        \end{itemize}
        Then $(\Gamma,\Out(H,\mc F^{(t)}))$ is cocycle-rigid.
        \item Let $H$ be a torsion-free group which is hyperbolic relative to a finite collection $\mc P$ of finitely generated subgroups. Assume that
        \begin{itemize}
            \item for every $P\in\mc P$, the pair $(\Gamma, P/Z(P))$ is cocycle-rigid, and
            \item for every $P\in \mc P$, the centralizer in $P$ of every nontrivial element is amenable.
        \end{itemize}
        Then $(\Gamma,\Out(H,\mc P^{(t)}))$ is cocycle-rigid.
    \end{enumerate}
\end{cor}

\bibliographystyle{alpha}  
\bibliography{biblio}
\end{document}